\documentclass[a4paper,10pt,reqno, english]{amsart}
\usepackage[utf8]{inputenc}
\usepackage[T1]{fontenc}
\usepackage{amsmath,amsthm}
\usepackage{amsfonts,amssymb,enumerate}
\usepackage{url,paralist}
\usepackage{mathtools,color}
\usepackage[colorlinks=true,urlcolor=blue,linkcolor=red,citecolor=magenta]{hyperref}
\usepackage{enumerate}
\usepackage{anysize}
\usepackage[arrow,curve,matrix,tips,2cell]{xy}
  \SelectTips{eu}{10} \UseTips
  \UseAllTwocells
\usepackage{lscape}
\usepackage{subfigure}

\theoremstyle{plain}
\newtheorem{theorem}{Theorem}[section]
\newtheorem{lemma}[theorem]{Lemma}
\newtheorem{claim}[theorem]{Claim}
\newtheorem{corollary}[theorem]{Corollary}

\newtheorem{conjecture}[theorem]{Conjecture}

\theoremstyle{definition}
\newtheorem{definition}[theorem]{Definition}

\newtheorem{remark}[theorem]{Remark}

\newcommand{\R}{\mathbb{R}}

\newcommand{\Z}{\mathbb{Z}}
\newcommand{\Q}{\mathbb{Q}}
\newcommand{\F}{\mathbb{F}}
\newcommand{\pt}{\mathrm{pt}}

\newcommand\sym{\mathfrak S}

\newcommand{\A}{\mathcal{A}}
\newcommand{\C}{\mathcal{C}}
\newcommand{\D}{\mathcal{D}}
\newcommand{\E}{\mathcal{E}}
\newcommand{\M}{\mathcal{M}}

\newcommand{\conv}{\operatorname{conv}}
\newcommand{\conf}{\operatorname{Conf}}

\newcommand{\id}{\operatorname{id}}

\newcommand{\colim}{\operatorname{colim}}

\newcommand{\relint}{\operatorname{relint}}

\newcommand{\emp}{\operatorname{EMP}}
\renewcommand{\int}{\operatorname{int}}
\newcommand{\hocolim}{\operatorname{hocolim}}
\newcommand{\supp}{\operatorname{supp}}

\begin{document}

\title{Cutting a part from many measures}


\author[Blagojevi\'c]{Pavle V. M. Blagojevi\'{c}}
\address{Institut f\" ur Mathematik, FU Berlin, Arnimallee 2, 14195 Berlin, Germany\hfill\break%
\mbox{\hspace{4mm}}Mathematical Institute SASA, Knez Mihailova 36, 11000 Beograd, Serbia}
\email{blagojevic@math.fu-berlin.de}

\author[Pali\'c]{Nevena Pali\'c}
\address{Institut f\" ur Mathematik, FU Berlin, Arnimallee 2, 14195 Berlin, Germany}
\email{nevenapalic@gmail.com}

\author[Sober\'on]{Pablo Sober\'on} 
\address{Baruch College, City University of New York, One Bernard Baruch Way, New York, NY 10010,\hfill\break%
\mbox{\hspace{4mm}}United States} 
\email{pablo.soberon-bravo@baruch.cuny.edu}

\author[Ziegler]{G\"unter M. Ziegler} 
\address{Institut f\" ur Mathematik, FU Berlin, Arnimallee 2, 14195 Berlin, Germany} 
\email{ziegler@math.fu-berlin.de}

\thanks{The research by Pavle V. M. Blagojevi\'{c} leading to these results has received funding from DFG via the Collaborative Research Center TRR~109 ``Discretization in Geometry and Dynamics'', and the grant ON 174024 of the Serbian Ministry of Education and Science.\newline\indent
The research by Nevena Pali\'c leading to these results has received funding from DFG via the Berlin Mathematical School.\newline\indent
The research by G\"unter M. Ziegler leading to these results has funding from DFG via the Collaborative Research Center TRR~109 ``Discretization in Geometry and Dynamics''.\newline\indent
This material is based upon work supported by the National Science Foundation under Grant No. DMS-1440140 while Blagojevi\'c, Pali\'c, and Ziegler were in residence at the Mathematical Sciences Research Institute in Berkeley CA, during the Fall of 2017. \newline\indent
The research of Pablo Sober\'on leading to these results has received funding from the National Science Foundation under Grant No. DMS-1851420.
}
\date{}

\begin{abstract}
Holmsen, Kyn\v{c}l and Valculescu recently conjectured that if a finite set $X$ with $\ell n$ points in $\mathbb{R}^d$ that is colored by $m$ different colors can be partitioned into $n$ subsets of $\ell$ points each, such that each subset contains points of at least $d$ different colors, then there exists such a partition of $X$ with the additional property that the convex hulls of the $n$ subsets are pairwise disjoint. 

We prove a continuous analogue of this conjecture, generalized so that each subset contains points of at least $c$ different colors, where we also allow $c$ to be greater than $d$. 
Furthermore, we give lower bounds on the fraction of the points each of the subsets contains from $c$ different colors.
For example, when $n\geq 2$, $d\geq 2$, $c\geq d$ with $m\geq n(c-d)+d$ are integers, and  $\mu_1, \dots, \mu_m$ are $m$ positive finite absolutely continuous measures on $\R^d$, we prove that there exists a partition of $\R^d$ into $n$ convex pieces which equiparts the measures $\mu_1, \dots, \mu_{d-1}$, and in addition every piece of the partition has positive measure with respect to at  least $c$ of the measures $\mu_1, \dots, \mu_m$.

\end{abstract}

\maketitle


\section{Introduction and statement of the main results}
\label{sec : Introduction and the statement of the main results }

The classical measure partition problems ask whether, for a given collection of measures of some Euclidean space, the ambient Euclidean space can be partitioned in a prescribed way so that each of the given measures gets cut into equal pieces. 

The first example of such a result is the well known ham-sandwich theorem, conjectured by Steinhaus and later proved by Banach.
It claims that given $d$ measures in $\R^d$, one can cut $\R^d$ by an affine hyperplane into two pieces so that each of the measures is cut into halves.
Motivated by the ham-sandwich theorem, Gr\"unbaum posed a more general hyperplane measure partition problem in 1960 \cite[Sec.\,4\,(v)]{Grunbaum1960}. 
He asked whether any given measure in the Euclidean space $\R^d$ can be cut by $k$ affine hyperplanes into $2^k$ equal pieces.
An even more general problem was proposed and considered by Hadwiger \cite{Hadwiger1966} and Ramos \cite{Ramos1996}:
Determine the minimal dimension $d$ such that for every collection of $j$ measures  on $\R^d$ there exists an arrangement of $k$ affine hyperplanes in $\R^d$ that cut all measures into $2^k$ equal pieces.
For a survey on the Gr\"unbaum--Hadwiger--Ramos hyperplane measure partition problem consult \cite{Blagojevic2015-02}.

Furthermore, in 2001 B\'ar\'any and Matou\v{s}ek \cite{Barany2001} considered partitions of measures on the sphere $S^2$ by fans with the requirement that each angle of the fan contains a prescribed proportion of every measure.

\medskip
In this paper, motivated by a conjecture of Holmsen, Kyn\v{c}l \& Valculescu \cite[Con.\,3]{Holmsen2017}, we consider many measures in a Euclidean space, and instead of searching for equiparting convex partitions we look for convex partitions that in each piece capture a positive amount from a (large) prescribed number of the given measures. 

\begin{definition}
\label{def:partition}
Let $d\geq 1$ and $n\geq 1$ be integers.
An ordered collection of closed subsets $(C_1, \dots, C_n)$ of $\R^d$ is called a \emph{partition} of $\R^d$ if
\begin{compactenum}[\rm \qquad(1)]
\item $\bigcup_{i=1}^n C_i = \R^d$,
\item $\int(C_i) \neq \emptyset$ for every $1\leq i\leq n$, and
\item $\int(C_i) \cap \int(C_j) = \emptyset$ for all $1\leq i< j\leq n$.	
\end{compactenum}
A partition $(C_1, \dots, C_n)$ is called \emph{convex} if all subsets $C_1, \dots, C_n$ are convex.
\end{definition}

Let $m\geq 1$, $n\geq 1$ , $c\geq 1$ and $d\geq 1$ be integers, and let $\mathcal{M}=(\mu_1,\ldots,\mu_m)$ be a collection of $m$ finite absolutely continuous measures in $\R^d$. 
Moreover, assume that $\mu_j(\R^d)>0$, for every $1 \leq j \leq m$.
For us a measure is an absolutely continuous measure if it is absolutely continuous with respect to the standard Lebesgue measure. 

\medskip
We are interested in the existence of a convex partition $(C_1, \dots, C_n)$ of $\R^d$ with the property that each set $C_i$ contains a positive amount of at least $c$ of the measures, that is
\[
\# \big\{ j : 1\leq j\leq m,\,\mu_j(C_i)>0\big\} \geq c, 
\]
for every $1\leq i\leq n$.
In the  case when the measures are given by finite point sets, we say that a point set $X\subseteq\R^d$ is in general position if no $d+1$ points from $X$ lie in an affine hyperplane in $\R^d$.
For the point set measures in general position Holmsen, Kyn\v{c}l and Valculescu  proposed the following natural conjecture \cite[Con.\,3]{Holmsen2017}. 

\begin{conjecture}[Holmsen, Kyn\v{c}l, Valculescu, 2017]
\label{conj:Holmsen}
Let $d\geq 2$, $\ell \geq 2$, $m\geq 2$ and $n\geq 1$ be integers with $m \geq d$ and $\ell \geq d$.
Consider a set $X\subseteq \R^d$ of $\ell n$ points in general position that is colored with at least $m$ different colors. 
If there exists a partition of the set $X$ into $n$ subsets of size $\ell$ such that each subset contains points colored by at least $d$ colors, then there exists such a partition of $X$ that in addition has the property that the convex hulls of the $n$ subsets are pairwise disjoint.   
\end{conjecture}

The conjecture was settled for $d=2$ in the same paper by Holmsen, Kyn\v{c}l and Valculescu \cite{Holmsen2017}.
On the other hand, if instead of finite collections of points one considers finite positive absolutely continuous measures in $\R^d$, Sober\'{o}n \cite{Soberon2012} gave a positive answer on splitting $d$ measures in $\R^d$ into convex pieces such that each piece has positive measure with respect to each of the measures. 
Moreover, he proved existence of convex partitions that equipart all measures.
A discretization of Sober\'on's result by Blagojevi\'{c}, Rote, Steinmeyer and Ziegler \cite{Blagojevic2017} gave a positive answer to Conjecture \ref{conj:Holmsen} in the case when $m=d$.
In addition, they were able to show that the set $X$ can be partitioned into $n$ subsets in such a way that all color classes are equipartitioned simultaneously.

\medskip
In this paper we prove several continuous results of a similar flavor, trying to come closer to a positive answer to Conjecture \ref{conj:Holmsen} in the case when $m\geq  d$.
The first of the three results is the following.

\begin{theorem}
\label{thm:geometric}
Let $d\geq 2$, $m\geq 2$, $n\geq 2$, and $c\geq d$ be integers.
If 
\[
m\geq n(c-d)+d,
\] 
then for every collection $\M=(\mu_1, \dots, \mu_m)$ of $m$ positive finite absolutely continuous measures on $\R^d$, there exists a partition of $\R^d$ into $n$ convex subsets $(C_1, \dots, C_n)$, which is a convex $n$-fan, such that each of the subsets has positive measure with respect to at least $c$ of the measures $\mu_1, \dots, \mu_m$. 
In other words, 
\[
\# \big\{ j : 1\leq j\leq m,\,\mu_j(C_i)>0\big\} \geq c
\]
for every $1\leq i \leq n$. 
\end{theorem} 

The following two theorems have stronger statements --- in Theorem \ref{thm:one measure} we additionally show that one of the measures can be equipartioned without changing the bound on $m$, and in Theorem \ref{thm:sum of measures} we prove that the sum of all the measures can be equipartitioned if we allow the number of measures  $m$ to increase.

\begin{theorem}
\label{thm:one measure}
Let $d\geq 2$, $m\geq 2$, and $c\geq d$ be integers, and let $n=p^k$ be a prime power.
If 
\[
m\geq n(c-d)+\frac{dn}{p}-\frac{n}{p}+1,
\] 
then for every collection $\M=(\mu_1, \dots, \mu_m)$ of $m$ positive finite absolutely continuous measures on $\R^d$, there exists a partition of $\R^d$ into $n$ convex subsets $(C_1, \dots, C_n)$ that equiparts the measure $\mu_m$ with the additional property that each of the subsets has positive measure with respect to at least $c$ of the measures $\mu_1, \dots, \mu_m$. 
In other words, 
\[
\mu_m(C_1) = \dots = \mu_m(C_n) = \frac{1}{n} \mu_m(\R^d), 
\]
and
\[
\# \big\{ j : 1\leq j\leq m,\,\mu_j(C_i)>0\big\} \geq c
\]
for every $1\leq i \leq n$. 
\end{theorem} 

\begin{theorem}
\label{thm:sum of measures}
Let $d\geq 2$, $m\geq 2$, and $c\geq d$ be integers, and let $n=p^k$ be a prime power.
If 
\begin{compactenum}[\rm\qquad(a)]
\item $n(c-1) \geq m$ and $\max\{m,n\}\geq n(c-d)+\tfrac{dn}{p}-\tfrac{n}{p}+n$, or
\item $n(c-1) < m$,
\end{compactenum}
then for every collection $\M=(\mu_1, \dots, \mu_m)$ of $m$ positive finite absolutely continuous measures on $\R^d$, there exists a partition of $\R^d$ into $n$ convex subsets $(C_1, \dots, C_n)$ that equiparts the sum of the measures $\mu = \mu_1 + \dots +\mu_m$ with the additional property that each of the subsets has positive measure with respect to at least $c$ of the measures $\mu_1, \dots, \mu_m$. 
In other words, 
\[
\mu(C_1) = \dots = \mu(C_n) = \frac{1}{n} \mu(\R^d) , 
\]
and
\[
\# \big\{ j : 1\leq j\leq m,\,\mu_j(C_i)>0\big\} \geq c
\]
for every $1\leq i \leq n$. 
\end{theorem}

Previous solutions for measure partition problems relied  on a variety of advanced methods from equivariant topology. 
Different configuration space/test map schemes (CS/TM schemes) related partition problems with the questions of {\em non}-existence of appropriately constructed equivariant maps from configuration spaces into suitable test spaces. 
For example, in the proof of the ham-sandwich theorem a sphere  with the antipodal action appears as a test space. 
The test space in the Gr\"unbaum--Hadwiger--Ramos hyperplane measure partition problem is again a sphere, but with an action of the sign permutation group, while the test space in the  B\'ar\'any and Matou\v{s}ek fan partition problem is a complement of an arrangement of linear subspaces equipped with an action of the Dihedral or generalized quaternion group.
In this paper the proof of Theorem \ref{thm:geometric} is elementary and it does not use any topology. 
However, the proofs of Theorem \ref{thm:one measure} and Theorem \ref{thm:sum of measures} rely on a novel CS/TM scheme presented in Theorem \ref{thm : CS-TM scheme one measure} and Theorem \ref{thm : CS-TM scheme sum of measures}:
For the first time the test space is the union of an arrangement of affine subspaces, equipped in this case  with an action of a symmetric group.

\medskip
Furthermore, even stronger measure partition result can be obtained directly without any use of advanced methods of equivariant topology.
We prove the following result with two similar arguments given in Section \ref{sec:proof referee - 01} and Section \ref{sec:proof referee - 02}. 

\begin{theorem}
\label{th:referee's}
Let $d\geq 2$, $m\geq 2$, $n\geq 2$, and $c\geq d$ be integers.
If 
\[
m=n(c-d)+d,
\] 
then for every collection $\M=(\mu_1, \dots, \mu_m)$ of $m$ positive finite absolutely continuous measures on $\R^d$, there exists a partition of $\R^d$ into $n$ convex subsets $(C_1, \dots, C_n)$ that equiparts the first $d-1$ measures $\mu_1,\dots,\mu_{d-1}$ with the additional property that each of the subsets has positive measure with respect to at least $c$ of the measures $\mu_1, \dots, \mu_m$. 
In other words, 
\[
\mu_k(C_1) = \dots = \mu_k(C_n) = \frac{1}{n} \mu_k(\R^d) 
\]
for every $1\leq k\leq d-1$, and
\[
\# \big\{ j : 1\leq j\leq m,\,\mu_j(C_i)>0\big\} \geq c
\]
for every $1\leq i \leq n$. 
\end{theorem}

\medskip
As a direct corollary of the proof of Theorem \ref{th:referee's} given in Section \ref{sec:proof referee - 02} we get the following straightening. 

\begin{corollary}
Let $d\geq 2$, $m\geq 2$, $n\geq 2$, and $c\geq d$ be integers.
If 
\[
m=n(c-d)+d,
\] 
then for every collection $\M=(\mu_1, \dots, \mu_m)$ of $m$ positive finite absolutely continuous measures on $\R^d$ with the property that
\[
\mu_d (\R^d) = \mu_{d+1}(\R^d) = \dots = \mu_m (\R^d)
\]
there exists a partition of $\R^d$ into $n$ convex subsets $(C_1, \dots, C_n)$ that equiparts the measures 
\[
\mu_1, \ \dots, \ \mu_{d-1}, \ \mu_d+\dots+\mu_m, \ \mu_1+\dots+\mu_m,  
\]
and has the additional property that each of the subsets has positive measure with respect to at least $c$ of the measures $\mu_1, \dots, \mu_m$.
In other words,
\[
\# \big\{ j : 1\leq j\leq m,\,\mu_j(C_i)>0\big\} \geq c
\]
for every $1\leq i \leq n$. 
\end{corollary}

\medskip
For theorems \ref{thm:geometric}, \ref{thm:one measure}, \ref{thm:sum of measures} and \ref{th:referee's} one can wonder: {\em Are the lower bounds on the number of the measures $m$ optimal} ?
We show that in the case when we require equipartition of $d-1$, out of $m$, measures the lower bound $m=n(c-d)+d$ from Theorem \ref{th:referee's} is optimal.  

\begin{theorem}
	\label{thm:optimal}
	Let $d \ge 1$, $n \ge 2$ and  $c \ge d$ be integers, and let $m = n(c-d)+d-1$.  
	There exists a collection of $m$ positive finite absolutely continuous measures in $\R^d$ such that for every cpartition of $\R^d$ into $n$ convex subsets $(C_1, \dots, C_n)$ that equiparte the first $d-1$ measures there is at least one part of the partition that has positive measure with respect to at most $c-1$ of the measures $\mu_1, \dots, \mu_m$.
\end{theorem}

\medskip
The technique used in the proof of Theorem \ref{th:referee's} can be further utilized to determine the fraction of the measures obtained by a partition.
We prove the following extension of Theorem \ref{th:referee's}.

\begin{theorem}
\label{th:referee's - 02}
Let $d\geq 2$, $m\geq 2$, $n\geq 2$, and $c\geq d$ be integers.
If 
\[
m= n(c-d)+d,
\] 
then for every collection $\M=(\mu_1, \dots, \mu_m)$ of $m$ positive finite absolutely continuous measures on $\R^d$, there exists a partition of $\R^d$ into $n$ convex subsets $(C_1, \dots, C_n)$ such that each of the subsets $C_1,\dots,C_n$ has at least an
\[
\varepsilon = \frac{1}{n \Big((n-1) \big( \lceil \frac{c}{d} \rceil -1\big) +1\Big)}  \ge \frac{d}{c n^2}
\]  
fraction of at least $c$ of the measures $\mu_1, \dots, \mu_m$.
In other words,
\[
\# \big\{ j : 1\leq j\leq m,\,\mu_j(C_i)\geq \varepsilon\mu_j(\R^d) \big\} \geq c
\]
for every $1\leq i \leq n$. 

\end{theorem}

Furthermore,  we prove that in some situations a fraction of measure each convex piece of a partition contains can be prescribed in advance.  
First, observe that if all the measures are equal we cannot hope to get more than a fraction of $\frac{1}{n}$ in many measures for each convex piece, (the smallest part prohibits this).  
Nevertheless, we prove that we can get as close as $\alpha = \frac1{n}$ as we want, provided we pay the price of using more measures.

\begin{theorem}
\label{th:referee's - 03}
Let $d\geq 2$, $m\geq 2$, $n\geq 2$, and $c\geq 2d$ be integers, and let $0<\alpha<\frac1{n}$ be a real number.
If 
\[
m \ge (c-d)\Big( \frac{1-\alpha}{\frac{1}{n}-\alpha}\Big) +d-1,
\] 
then for every collection $\M=(\mu_1, \dots, \mu_m)$ of $m$ positive finite absolutely continuous measures on $\R^d$, there exists a partition of $\R^d$ into $n$ convex subsets $(C_1, \dots, C_n)$ such that each of the subsets $C_1,\dots,C_n$ has at least $\alpha$ fraction in at least $c$ of the measures $\mu_1, \dots, \mu_m$.
In other words,
\[
\# \big\{ j : 1\leq j\leq m,\,\mu_j(C_i)\geq \alpha\mu_j(\R^d) \big\} \geq c
\]
for every $1\leq i \leq n$. 
Moreover, if $\frac{1-\alpha}{\frac{1}{n}-\alpha}$ is an integer, then it suffices to have
\[
m \ge (c-d)\Big( \frac{1-\alpha}{\frac{1}{n}-\alpha}\Big),
\]
measures.
\end{theorem}

The reader may notice that, since $\frac{1-\alpha}{\frac{1}{n}-\alpha} > n$ if $\alpha > 0$, if we make $\alpha \to 0$ we recover the same number of measures as Theorem \ref{thm:geometric} and Theorem \ref{th:referee's} for $c \ge 2d$.  In Theorem \ref{thm:geometric} we have a stronger control on the kind of partitions obtained and in Theorem \ref{th:referee's} we can also guarantee the equipartition of $d-1$ measures.
 
\medskip

The rest of the paper is organized as follows. 
The proofs of Theorem \ref{thm:one measure} and Theorem \ref{thm:sum of measures} run in parallel and follow CS/TM schemes that are given in Section \ref{sec:From partitions to equivariant topology}. 
The topological results about non-existence of equivariant maps are proved in Section \ref{sec:Non-existence of equivariant maps}. 
Finally, the proofs of theorems \ref{thm:geometric}, \ref{thm:one measure}, \ref{thm:sum of measures}, \ref{th:referee's}, \ref{thm:optimal}, \ref{th:referee's - 02} and \ref{th:referee's - 03} are given in Section \ref{sec:proofs}.
Note that the proofs of theorems \ref{thm:geometric}, \ref{th:referee's}, \ref{thm:optimal}, \ref{th:referee's - 02}, and \ref{th:referee's - 03} can be read independently of the previous sections.

\subsection*{Acknowledgement} 
We are grateful to Jonathan Kliem, Johanna K. Steinmeyer and Roman Karasev for many useful observations and suggestions.

\section{Existence of a partition from non-existence of a map}
\label{sec:From partitions to equivariant topology}

In this section we develop CS/TM schemes that relate the existence of convex partitions from theorems \ref{thm:one measure} and \ref{thm:sum of measures} with the non-existence of particular equivariant maps. These two CS/TM schemes are very similar to each other.

\subsection{Existence of an equipartition of one measure from non-existence of a map}
\label{sec:From partitions to equivariant topology I}

Let $d\geq 2$, $m\geq 2$, $n\geq 1$, and $c\geq 2$ be integers, and let $\M=(\mu_1, \dots, \mu_m)$ be a collection of finite absolutely continuous measures on $\R^d$.
Throughout the paper we assume that $m \geq c$, since it is a requirement that naturally comes from the mass partition problem.
Following notation from \cite{Blagojevic2014}, let $\emp(\mu_m,n)$ denote the space of all convex partitions of $\R^d$ into $n$ convex pieces $(C_1, \dots, C_n)$ that equipart the measure $\mu_m$, as studied in \cite{Leon2015}, that is
\[
\mu_m(C_1) = \dots = \mu_m(C_n) = \frac{1}{n} \mu_m(\R^d).
\]

Now define a continuous map $f_{\M}\colon \emp(\mu_m,n) \longrightarrow \R^{(m-1) \times n}\cong (\R^{m-1})^n$ as 
\[
(C_1,\dots,C_n) \longmapsto
\left(
\begin{matrix}
\mu_1(C_1) & \mu_1(C_2) & \dots & \mu_1(C_n) \\
\mu_2(C_1) & \mu_2(C_2) & \dots & \mu_2(C_n) \\
\vdots     & \vdots     & \ddots & \vdots \\
\mu_{m-1}(C_1) & \mu_{m-1}(C_2) & \dots & \mu_{m-1}(C_n) \\
\end{matrix} \right).
\]
The symmetric group $\sym_n$ acts on $\emp(\mu_m,n)$ and $(\R^{m-1})^n$ as follows
\[
\pi\cdot (C_1,\ldots,C_n) = (C_{\pi(1)},\ldots,C_{\pi(n)})
\qquad\text{and}\qquad
\pi\cdot (Y_1,\ldots,Y_n)= (Y_{\pi(1)},\ldots,Y_{\pi(n)}),
\]
where $(C_1,\ldots,C_n)\in \emp(\mu_m,n)$, $(Y_1,\ldots,Y_n)\in (\R^{m-1})^n$, and $\pi\in\sym_n$.
These actions are introduced in such a way that the map $f_{\M}$ becomes an $\sym_n$-equivariant map.
The image of the map $f_{\M}$ is a subset of the affine set $V \subseteq \R^{(m-1) \times n}\cong (\R^{m-1})^n$ given by 
\[
V= \Big\{ (y_{jk})   \in \R^{(m-1) \times n} \ : \
\sum_{k=1}^n y_{jk} = \mu_j(\R^d)\textrm{ for every }1\leq j\leq m-1 \Big\} \cong \R^{(m-1)\times(n-1)}.	
\]
Consequently, we can assume that $f_{\M}\colon \emp(\mu_m,n) \longrightarrow V\subseteq \R^{(m-1) \times n}$.

\medskip
Let $1\leq i\leq n$, and let $I \subseteq [m-1]$ be a subset of cardinality $|I|=m-c+1$, where $[m-1]$ denotes the set of integers $\{1,2,\dots, m-1\}$.
Consider  the subspace $L_{i,I}$ of $V$ given by 
\[
 L_{i,I}: = \big\{ (y_{jk})   \in V : y_{r,i}=0 \textrm{ for every } r \in I \big\},
\]
and the associated arrangement 
\begin{eqnarray}
\label{eq:A one measure}
\A =\A(m,n,c):= \left\lbrace L_{i,I} : 1\leq i \leq n, \ I \subseteq [m-1], \ |I|=m-c+1 \right\rbrace.
\end{eqnarray}
The arrangement $\A$ is an $\sym_n$-invariant affine arrangement in $\R^{(m-1) \times n}$, meaning that $\pi\cdot  L_{i,I}\in\A$ for every $\pi\in\sym_n$.
Now we explain the key property of the arrangement $\A$.
Let $(C_1, \dots, C_n)$ be a convex partition of $\R^d$ tat equiparts $\mu_m$ with a property that at least one of the subsets $C_1, \dots, C_n$ has positive measure with respect to at most $c-1$ of the measures $\mu_1, \dots, \mu_m$, which means that $(C_1, \dots, C_n)$ is not a partition we are searching for. 
Since, by construction $\mu_m(C_i)>0$ for every $1 \leq i \leq n$, it follows that at least one of the subsets $C_1, \dots, C_n$ has positive measure with respect to at most $c-2$ of the measures $\mu_1, \dots, \mu_{m-1}$.
Then there is a column of the matrix $f_{\M}(C_1, \dots, C_n)\in V\subseteq \R^{(m-1)\times n}$ with at most $c-2$ positive coordinates. 
In other words, there is a column of the matrix $f_{\M}(C_1, \dots, C_n)$ with at least $m-c+1$ zeros, and consequently the matrix $f_{\M}(C_1, \dots, C_n)$ is an element of the union $\bigcup \A:=\bigcup_{L_{i,I} \in \A} L_{i,I}$ of the arrangement $\A$.

\medskip
Let us now assume that for integers $d\geq 2$, $m\geq 2$, $n\geq 1$, and $c\geq 1$, there exists a collection $\M=(\mu_1, \dots, \mu_m)$ of absolutely continuous positive finite measures in~$\R^d$ such that in every convex partition $(C_1, \dots, C_n)$ of $\R^d$ that equiparts $\mu_m$ there is at least one subset $C_k$ that does not have positive measure with respect to at least $c$ measures, or equivalently it has measure zero with respect to at least $m-c+1$ of the measures $\mu_1, \dots, \mu_m$.
 Consequently, $f_{\M}(C_1, \dots, C_n) \in \bigcup \A$ for every convex partition $(C_1, \dots, C_n)$ of $\R^d$ that equiparts the measure $\mu_m$. In particular, this means that the $\sym_n$-equivariant map $f_{\M}$ factors as follows
 \[
 \xymatrix{
 \emp(\mu_m,n)\ar[rr]^-{f_{\M}}\ar[rd]_{f_{\M}'} & & V\\
  & \bigcup \A(m,n,c) , \ar[ru]_-{i} &
 }
 \] 
 where $i\colon \bigcup \A\longrightarrow V$ is the inclusion and $f_{\M}'\colon\emp(\mu_m,n)\longrightarrow \bigcup\A$ is the $\sym_n$-equivariant map obtained from $f_{\M}$ by restricting the codomain.
Thus, we have proved the following theorem. 

\begin{theorem}
	\label{thm : CS-TM scheme one measure}
	Let $d\geq 2$, $m\geq 2$, $n\geq 1$, and $c\geq 2$ be integers, and let $\M=(\mu_1, \dots , \mu_m)$ be be a collection of absolutely continuous positive finite measures on $\R^d$.
	If there is no $\sym_n$-equivariant map
	\[
	  \emp(\mu_m,n) \longrightarrow \bigcup\A(m,n,c),
	\]
	then there exists a convex partition $(C_1, \dots, C_n)$ of $\R^d$ that equiparts the measure $\mu_m$ with the additional property that  each of the  subsets $C_i$ has positive measure with respect to at least $c$ of the measures $\mu_1, \dots, \mu_m$, that is 
\[
\mu_m(C_1) = \dots = \mu_m(C_n) = \frac{1}{n} \mu_m(\R^d), 
\]
and
\[
\# \big\{ j : 1\leq j\leq m,\,\mu_j(C_i)>0\big\} \geq c
\] 
for every $1\leq i \leq n$.
\end{theorem} 

\subsection{Existence of an equipartition of the sum of  measures from non-existence of a map}
\label{sec:From partitions to equivariant topology II}

As we have already mentioned, the CS/TM scheme needed for proving Theorem \ref{thm:sum of measures} is very similar to the one presented in Section \ref{sec:From partitions to equivariant topology I}. 
Nevertheless, it will be developed separately here.

\medskip
Let $d\geq 2$, $m\geq 2$, $n\geq 1$, and $c\geq 2$ be integers, and let $\M=(\mu_1, \dots, \mu_m)$ be a collection of absolutely continuous positive finite measures on $\R^d$. 
Denote by $\mu$ the sum of the measures $\mu_1, \dots, \mu_m$, that is $\mu := \sum_{j=1}^m \mu_j$.

Similarly as in Section \ref{sec:From partitions to equivariant topology I}, we define a continuous map $\widetilde{f}_{\M}\colon \emp(\mu,n) \longrightarrow \R^{m \times n}$ as 
\[
(C_1,\dots,C_n) \longmapsto
\left(
\begin{matrix}
\mu_1(C_1) & \mu_1(C_2) & \dots & \mu_1(C_n) \\
\mu_2(C_1) & \mu_2(C_2) & \dots & \mu_2(C_n) \\
\vdots     & \vdots     & \ddots & \vdots \\
\mu_{m}(C_1) & \mu_{m}(C_2) & \dots & \mu_{m}(C_n) \\
\end{matrix} \right),
\]
where the domain of the map $\widetilde{f}_{\M}$ is the space of all convex partitions of $\R^d$ that equipart the measure $\mu$. 
The map $\widetilde{f}_{\M}$ is $\sym_n$-equivariant by construction.
Furthermore, the image of the map $\widetilde{f}_{\M}$ is a subset of the affine set $\widetilde{V} \subseteq \R^{m \times n}$ given by 
\begin{eqnarray*}
\widetilde{V}= \Big\{ (y_{jk}) \in \R^{m \times n} \ : \
& \sum_{k=1}^n y_{jk} = \mu_j(\R^d)& \textrm{ for every }1\leq j\leq m,\\
& \sum_{j=1}^m y_{jk} = \frac{1}{n} \mu(\R^d)& \textrm{ for every } 1\leq k \leq n \Big\}.	
\end{eqnarray*}

\medskip
Now we define an affine arrangement that resembles the arrangement $\A$ from Section \ref{sec:From partitions to equivariant topology I}.
Let $1\leq i\leq n$, and let $I \subseteq [m]$ be a subset of cardinality $|I|=m-c+1$.
Consider  the subspace $\widetilde{L}_{i,I}$ of $\widetilde{V}$ given by 
\[
 \widetilde{L}_{i,I}: = \big\{ (y_{jk})   \in \widetilde{V} : y_{r,i}=0 \textrm{ for every } r \in I \big\},
\]
and the associated $\sym_n$-invariant arrangement 
\begin{eqnarray}
\label{eq:A sum of measures}
\widetilde{\A} =\widetilde{\A}(m,n,c):= \left\lbrace \widetilde{L}_{i,I} : 1\leq i \leq n, \ I \subseteq [m], \ |I|=m-c+1 \right\rbrace.
\end{eqnarray}

Following the steps from Section \ref{sec:From partitions to equivariant topology I}, we study the key property of the arrangement $\widetilde{\A}$. Let $(C_1, \dots, C_n)$ be a convex partition of $\R^d$ that does not satisfy the property asked in Theorem \ref{thm:sum of measures}. 
More precisely, assume that for some $i$ the subset $C_i$ has positive measure with respect to at most $c-1$ of the measures $\mu_1, \dots, \mu_m$. 
This means that the $i$-th column of the matrix $\widetilde{f}_{\M}(C_1, \dots, C_n) \in \R^{m \times n}$ has at least $m-c+1$ zeros. 
In other words, $\widetilde{f}_{\M}(C_1, \dots, C_n) \in \bigcup\widetilde{\A}$. Therefore, we have obtained the following theorem.

\begin{theorem}
	\label{thm : CS-TM scheme sum of measures}
	Let $d\geq 2$, $m\geq 2$, $n\geq 1$, and $c\geq 2$ be integers, and let $\M=(\mu_1, \dots ,\mu_m)$ be a collection of absolutely continuous positive finite measures on $\R^d$.
	If there is no $\sym_n$-equivariant map
	\[
	  \emp(\mu,n) \longrightarrow \bigcup\widetilde{\A}(m,n,c),
	\]
	then there exists a convex partition $(C_1, \dots, C_n)$ of $\R^d$ that eqiparts the measure $\mu = \mu_1 + \dots +\mu_m$ with the additional property that  each of the  subsets $C_i$ has positive measure with respect to at least $c$ of the measures $\mu_1, \dots, \mu_m$, that is 
\[
\mu(C_1) = \dots = \mu(C_n) = \frac{1}{n} \mu(\R^d),
\]
and
\[
\# \big\{ j : 1\leq j\leq m,\,\mu_j(C_i)>0\big\} \geq c
\] 
for every $1\leq i \leq n$.
\end{theorem}

\section{Non-existence of the equivariant maps}
\label{sec:Non-existence of equivariant maps}

This section is devoted to the proof of (non)-existence of equivariant maps from the space of regular convex partitions to appropriate affine arrangements. 
In Section \ref{sec:Non-existence of equivariant maps I} we consider the existence of an $\sym_n$-equivariant maps $ \emp(\mu_m,n) \longrightarrow \bigcup\A(m,n,c)$, whereas in Section \ref{sec:Non-existence of equivariant maps II} we focus on the existence $\sym_n$-equivariant maps $ \emp(\mu,n) \longrightarrow \bigcup\widetilde{\A}(m,n,c)$ will be considered for different values of integer parameters $d$, $m$, $n$ and $c$.

\subsection{Non-existence of an $\sym_n$-equivariant map $\emp(\mu_m,n) \longrightarrow \bigcup\A(m,n,c)$}
\label{sec:Non-existence of equivariant maps I}

In order to prove the (non-)existence of an $\sym_n$-equivariant map
\[
	  \emp(\mu_m,n) \longrightarrow \bigcup\A(m,n,c),
\]
we first construct various equivariant maps and prove a few auxiliary lemmas.
In the following we use particular tools from the theory of homotopy colimits; for further details on these methods consult for example \cite{Bousfield1972}, \cite{Welker1999}, or \cite{Sundaram1997}.

\medskip
 Let $X$ be a topological space and let $n\geq 1$ be an integer. 
The {\em ordered configuration space} $\conf(X,n)$ of $n$ ordered pairwise distinct points of $X$ is the space
\begin{eqnarray*}
\conf(X,n):= \{ (x_1, \dots, x_n) \in X^n \mid x_i \neq x_j \textrm{ for all } 1\leq i< j \leq n\}.
\end{eqnarray*}
It was shown in \cite[Sec.\,2]{Blagojevic2014} that a subspace of $\emp(\mu_m,n)$ consisting only of regular convex partitions can be parametrized by the configuration space $\conf(\R^d,n)$.
In particular, we have the following lemma.

\begin{lemma}
	\label{lemma : alpha}
	There exists an $\sym_n$-equivariant map
	\[
	\alpha\colon \conf(\R^d,n) \longrightarrow \emp(\mu_m,n).
	\]
\end{lemma}
 
\medskip
Let $P:=P(\A)$ denote the intersection poset of the arrangement $\A=\A(m,n,c)$, ordered by the reverse inclusion.
The elements of the poset $P$ are non-empty intersections of subspaces in $\A$, thus they are of the form 
\[
p_\Lambda := \bigcap_{(i,I) \in \Lambda} L_{i,I} =\big\{ (y_{jk})  \in V\subseteq \R^{(m-1) \times n} \ : \ y_{ji}=0, \textrm{ for all }1\leq i\leq n\textrm{ and }  j \in I_i\big\},
\]
where $\Lambda \subseteq [n] \times \binom{[m-1]}{m-c+1}$ and $I_i := \bigcup_{(i,I)\in \Lambda} I$.
Observe that sets $I_i$ can be empty.
Alternatively, each poset element $p_\Lambda$ can be presented as an $(m-1) \times n$ matrix $(a_{jk}) $, where $a_{jk}=0$ if and only if $j\in I_k$. 
In other words, a coordinate $a_{jk}$ in the matrix presentation of $p_\Lambda$ equals zero if and only if $y_{jk}=0$ for every element $(y_{jk}) \in p_\Lambda$. 
An example of the poset $P(\A)$ for parameters $n=2, m=4$ and $c=3$ is shown in Figure \ref{fig:poset}.

\begin{figure}[h]
\begin{center}
\includegraphics[width=0.5\textwidth]{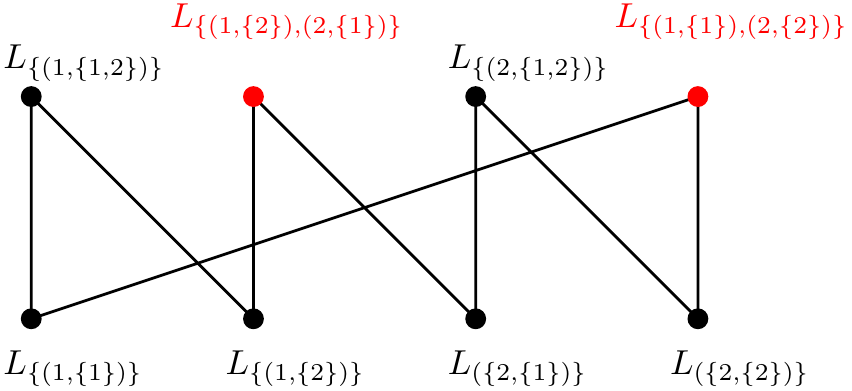}
\caption{Hasse diagram of the poset $P(\A(4,2,3))$.}
\label{fig:poset}
\end{center}
\end{figure}

Let $\C$ be the $P$-diagram that corresponds to the arrangement $\A=\A(m,n,c)$, that is $\C(p_\Lambda):= p_\Lambda$ and 
$\C(p_{\Lambda'}\supseteq p_{\Lambda''})\colon p_{\Lambda''}\longrightarrow p_{\Lambda'}$ is the inclusion, see \cite[Sec.\,2.1]{Welker1999}.
The Equivariant Projection Lemma \cite[Lem.\,2.1]{Sundaram1997} implies the following.
 
\begin{lemma}
	\label{lemma : beta}
	The projection map  
	\[
	\hocolim_{P(\A)}\C \longrightarrow \colim_{P(\A)}\C  = \bigcup\A
	\]
	is an $\sym_n$-equivariant homotopy equivalence.
	In particular, there exists an $\sym_n$-equivariant map
	\[
	\beta\colon \bigcup\A\longrightarrow \hocolim_{P(\A)}\C.
	\]
\end{lemma}

\medskip
Now, let $Q$ be the face poset of the $(n-1)$-dimensional simplex where the order is given by the inclusion.
Define the monotone map $\varphi\colon P \longrightarrow Q$ by 
\[
\varphi \left( p_\Lambda \right) := \left\lbrace i \in [n] :  (i,I) \in \Lambda \text{ for some }I \subseteq [m-1] \right\rbrace.
\]
Thus, $\varphi$ maps an element $p_\Lambda$ to the set of indices of its columns that contain zeros.
It is important to notice that $\varphi$ does not have to be surjective, and therefore we set $Q':=\varphi(P)\subseteq Q$.

\medskip
Next we consider the homotopy pushdown $\D$ of the diagram $\C$ along the map $\varphi$ over $Q'$, see \cite[Sec.\,3.2]{Welker1999}. 
This means that for $q \in Q'$
\begin{eqnarray*}
\D(q):=\hocolim_{\varphi^{-1}(Q'_{\geq q})} \C|_{\varphi^{-1} (Q'_{\geq q})} \simeq \Delta(\varphi^{-1}(Q'_{\geq q})) ,
\end{eqnarray*}
and for every $q\geq r$ in $Q'$ the map $\D(q\geq r)\colon \D(q)\longrightarrow \D(r)$
is the corresponding inclusion. 
The next result follows from the Homotopy Pushdown Lemma \cite[Prop.\,3.12]{Welker1999} adapted to the equivariant setting. 

\begin{lemma}
	\label{lemma : gamma}
	There is an $\sym_n$-equivariant homotopy equivalence 
	\[
	\hocolim_{Q'}\D \longrightarrow \hocolim_{P(\A)}\C .
	\]
	In particular, there exists an $\sym_n$-equivariant map
	\[
	\gamma\colon \hocolim_{P(\A)}\C\longrightarrow \hocolim_{Q'}\D.
	\]
\end{lemma}

\medskip
We introduce another $Q'$-diagram $\E$ by setting for $q\in Q'$ that
\begin{eqnarray*}
\E(q): = \begin{cases}
		\D(\hat{1})\,\simeq\, \Delta(\varphi^{-1}(\{\hat{1}\})), & \textrm{ if } q=\hat{1}\in Q' \textrm{ is the maximum of } Q, \\
		\pt, & \textrm{ otherwise,}
		\end{cases}
\end{eqnarray*}
and for every $q\geq r$ in $Q'$ the map $\E(q\geq r)$ to be the constant map.
In addition, we define a morphism of diagrams $(\Psi,\psi)\colon \D\longrightarrow\E$, where $\psi\colon Q'\longrightarrow Q'$ is the identity map, and $\Psi(q)\colon \D(q)\longrightarrow \E(q)$ is the identity map when $q$ is the maximal element, and constant map otherwise.
The morphism $(\Psi,\psi)$ of diagrams induces an $\sym_n$-equivariant map between associated homotopy colimits, consult \cite[Sec.\,3]{Welker1999}. 
Thus, we have established the following.

\begin{lemma}
	\label{lemma : delta}
	There exists an $\sym_n$-equivariant map
	\[
	\delta\colon \hocolim_{Q'}\D\longrightarrow \hocolim_{Q'}\E.
	\]
\end{lemma}

\medskip	
In the final lemma we describe the $\hocolim_{Q'}\E$ up to an $\sym_n$-equivariant homotopy.
First note that if $q,r \in Q$ are such that $q \geq r$ and $q\in Q'$, then $r\in Q'$. In particular, if $\hat{1} \in Q'$, then $Q'=Q$, where $\hat{1}$ is the maximum of $Q$.
\begin{lemma}
\label{lemma : hocolim E}
~~
\begin{compactenum}[\rm\qquad (i)]
\item If $\hat{1}\in Q'$, that is $Q'=Q$, then there exists an $\sym_n$-equivariant homotopy equivalence
\[
\hocolim_{Q'}\E \,\simeq\, \Delta (Q'{\setminus}\{\hat{1}\}) * \Delta(\varphi^{-1}(\hat{\{1\}}))
\]
where $\hat{1}$ is the maximum of $Q$, and $\dim\big( \Delta(\varphi^{-1}(\{\hat{1}\}))\big) =nc-m-2n+1$.
In particular, there exists an $\sym_n$-equivariant map
\[
\eta\colon \hocolim_{Q'}\E\longrightarrow \Delta (Q'{\setminus}\{\hat{1}\}) * \Delta(\varphi^{-1}(\{\hat{1}\})).
\]

\item If $\hat{1}\notin Q'$ then there exists an $\sym_n$-equivariant homotopy equivalence
\[
\hocolim_{Q'}\E\simeq \Delta (Q'),
\]
where $\dim ( \Delta (Q'))\leq n-2$.
In particular, there exists an $\sym_n$-equivariant map
\[
\eta\colon \hocolim_{Q'}\E\longrightarrow \Delta (Q').
\]
\end{compactenum}

\end{lemma}
\begin{proof}
{\rm (i)}
Let us first consider the case when $\hat{1}\in Q'$.
Then, since all the maps of the diagram $\E$ are constant maps, the Wedge Lemma \cite[Lem.\,4.9]{Welker1999} yields a homotopy equivalence 
\[
\hocolim_{Q'} \E \ \simeq \ \bigvee_{q \in Q'} (\Delta(Q'_{<q}) * \E(q)) \vee \Delta (Q') \ \simeq \  \Delta (Q'{\setminus}\{\hat{1}\}) * \Delta(\varphi^{-1}(\{\hat{1}\})).
\]
Here we use that $\Delta(Q')\simeq\pt$ because $\Q'$ has the maximum.
Furthermore, since for $q\neq \hat{1}$ all the spaces $\E(q)$ are points, this homotopy equivalence is an $\sym_n$-equivariant homotopy equivalence.

The poset $\varphi^{-1}(\{\hat{1}\})$ consists of all points $p_\Lambda \in P$ that correspond to matrices which have zeros in all columns. 
Since it is a subposet of $P(\A)$, every element of $\varphi^{-1}(\{\hat{1}\})$ must contain at least $m-c+1$ zeros in each column and at most $n-1$ zeros in each row. The partial order is given by 
\[
p_\Lambda \leq p_{\Lambda'}  \ \Longleftrightarrow \ (\forall j\in [m-1]) \ (\forall k \in [n]) \ a_{jk}=0 \Rightarrow a'_{jk}=0,
\]
where $p_\Lambda = (a_{jk}) $ and $p_{\Lambda'} = (a'_{jk}) $. 
Maximal chains in the poset $\varphi^{-1}(\hat{1})$ can be obtained by removing zeros from a maximal element $p_\Lambda$ one by one, taking care that there must be at least $m-c+1$ zeros in each column. 
Maximal elements of $\varphi^{-1}(\hat{1})$ have exactly one non-zero element in each row, thus $(m-1)(n-1)$ zeros. 
Since $\hat{1}\in Q'$ the minimal elements of the poset $\varphi^{-1}(\hat{1})$ have $m-c+1$ zeros in each column, thus $n(m-c+1)$ zeros. 
Therefore, the length of a maximal chain in $\varphi^{-1}(\hat{1})$, and consequently the dimension of its order complex $\Delta(\varphi^{-1}(\hat{1}))$, is $nc-m-2n+1$.
In particular, we obtained that when $\hat{1}\in Q'$ then $nc-m-2n+1\geq 0$, or equivalently $n(c-2)+1\geq m$.

{\rm (ii)} Let $\hat{1}\notin Q'$. 
Then using the inclusion-exclusion principle it is not hard to see that $n(c-2)+1< m$.
Again, the Wedge Lemma \cite[Lem.\,4.9]{Welker1999} yields a homotopy equivalence 
\[
\hocolim_{Q'} \E \ \simeq \ \bigvee_{q \in Q'} (\Delta(Q'_{<q}) * \E(q)) \vee \Delta (Q') \ \simeq \  \Delta (Q' ),
\]
since now all the spaces $\E(q)$ are points for $q\in Q'$.

From the assumption $\hat{1}\notin Q'$ we get that $Q'\subseteq Q{\setminus}\{\hat{1}\}$ and consequently $\Delta(Q')\subseteq \Delta (Q{\setminus}\{\hat{1}\})$.
On the other hand $\Delta (Q{\setminus}\{\hat{1}\})$ is homeomorphic with the boundary of an $(n-1)$-dimensional simplex and so $\dim ( \Delta (Q'))\leq n-2$.
\end{proof}

In the example for parameters $n=2, m=4, c=3$, the poset $\varphi^{-1}(\{\hat{1}\})$ consists of two points $L_{\{(1,\{2\}),(2,\{1\})\}}$ and $L{\{(1,\{1\}),(2,\{2\})\}}$ with no relations between them, as shown in red in Figure \ref{fig:poset}.

\medskip
Now we have assembled all the ingredients for the proof of the central result about the non-existence of an $\sym_n$-equivariant map $\emp(\mu_m,n) \longrightarrow \bigcup\A$.
 
\begin{theorem}
\label{thm:no map one measure}
Let $d\geq 2$, $m\geq 2$, and $c\geq 2$ be integers, and let $n=p^k$ be a prime power.
If 
\[
m\geq n(c-d)+\frac{dn}{p}-\frac{n}{p}+1,
\]
then there is no continuous $\sym_n$-equivariant map
\begin{equation}
	\label{eq : map that should not exist - 01}
	 \emp(\mu_m,n) \longrightarrow \bigcup\A(m,n,c),
\end{equation}
where $\mu_m$ is a finite absolutely continuous measure on $\R^d$, and the affine arrangement $\A(m,n,c)$ is as defined in \eqref{eq:A one measure}.
\end{theorem}
\begin{proof}
Let $n=p^k$ be a prime power.
Denote by $G\cong (\Z/p)^k$ a subgroup of the symmetric group $\sym_n$ given by the regular embedding $\mathrm{(reg)} \colon G \longrightarrow \sym_n$,~for more details see for example \cite[Ex.\,III.2.7]{Adem2004}.

\medskip
In order to prove the non-existence of an $\sym_n$-equivariant map \eqref{eq : map that should not exist - 01}, we proceed by contradiction.
Let $f\colon \emp(\mu_m,n) \longrightarrow \bigcup\A(m,n,c)$ be a continuous $\sym_n$-equivariant map.
Then from Lemmas \ref{lemma : alpha}, \ref{lemma : beta}, \ref{lemma : gamma}, \ref{lemma : delta} and \ref{lemma : hocolim E} we get the following composition of $\sym_n$-equivariant maps
\[
\xymatrix{
 \emp(\mu_m,n) \ar[r]^-{f} &  \bigcup\A\ar[r]^-{\beta} & \hocolim_{P(\A)}\C\ar[r]^-{\gamma} & \hocolim_{Q'}\D\ar[r]^-{\delta} & \hocolim_{Q'}\E\ar[d]_{\eta}  \\
\conf(\R^d,n)\ar[u]^{\alpha}\ar@{-->}[rrrr]^-{g:=\eta\circ\delta\circ\gamma\circ\beta\circ f\circ\alpha} &  &  &  &  X,
}
\]
where
\[
X:=
\begin{cases}
	\Delta(Q'{\setminus}\{\hat{1}\}) * \Delta(\varphi^{-1}(\{\hat{1}\})), & \text{if } \hat{1}\in Q',\\
	\Delta(Q'),  &\text{if } \hat{1}\notin Q'.
	\end{cases} 
\]
Thus, the existence of an $\sym_n$-equivariant map $f\colon \emp(\mu_m,n)\longrightarrow \bigcup\A$ implies the existence of an $\sym_n$-equivariant map $g\colon \conf(\R^d,n)\longrightarrow C$.
We will reach contradiction with the assumption that the map $f$ exists by proving that the map $g$ cannot exist.
More precisely, we will show that there cannot exist a $G$-equivariant map
\begin{equation}
		\label{eq : map that should not exist - 02}
		\conf(\R^d,n)\longrightarrow X.
\end{equation}

\medskip
Our argument starts with a continuous $\sym_n$ and also $G$-equivariant map $g\colon \conf(\R^d,n)\longrightarrow X$.
The map $g$ induces a morphism between Borel construction fibrations:
	\[
	\xymatrix{
	EG\times_G  \conf(\R^d,n)\ar[rr]^{\id\times_G g}\ar[d]_{\lambda} & & EG\times_G X\ar[d]_{\rho} \\
	BG\ar[rr]^-{\id}  & & BG,
	}
	\]
which in turn induces a morphism between corresponding Serre spectral sequences $E^{*,*}_*(g)\colon E^{*,*}_*(\rho)\longrightarrow E^{*,*}_*( \lambda )$.
The crucial property of the morphism $E^{*,*}_*(g)$ we  use is that $E^{*,0}_2(g)=\id$.
A contradiction with the assumption that there is a map $g$ is going to be obtained from an analysis of the morphism $E^{*,*}_*(g)$.
For that we first describe the spectral sequences $E^{*,*}_*(\lambda)$ and $E^{*,*}_*(\rho)$.

\medskip
The Serre spectral sequence of the fibration
	\[
	\xymatrix{
	\conf(\R^d,n)\ar[r] & EG\times_G \conf(\R^d,n)\ar[r] & BG
	}
	\]
has the $E_2$-term given by
	\[
	E^{i,j}_2(\lambda)=H^i(BG;\mathcal{H}^j(\conf(\R^d,n);\F_p))\cong H^i(G;H^j(\conf(\R^d,n);\F_p)).
	\]
	Here $H^i(BG;\mathcal{H}^j(Y;\F_p))$ denotes the cohomology of $BG$ with local coefficients in $H^j(Y;\F_p)$ determined by the action of the fundamental group of the base space $\pi_1(BG)\cong G$.
	The second description uses the fact that cohomology of the classifying space $BG$ of the group $G$ is by definition the cohomology of the group $G$ with coefficients in the $G$-module $H^j(\conf(\R^d,n);\F_p)$.  
	For more details on the cohomology with local coefficients consult for example \cite[Sec.\,3.H]{Hatcher2002}.
	The spectral sequence $E^{*,*}_*(\lambda)$ was completely determined in the case $k=1$, i.e., $n=p$ a prime, by Cohen \cite[Thm.\,8.2]{Cohen1976LNM533} and recently in \cite[Thm.\,6.1]{Blagojevic2015}.  
	A partial description of $E^{*,*}_*(\lambda)$ in the case $k\geq 2$ was given in \cite[Thm.\,6.3 and Thm.\,7.1]{Blagojevic2015}.
	In particular, for $k=1$
	\begin{equation}
	\label{eq : ss - 01}
	E^{*,*}_2(\lambda) \ \cong \ E^{*,*}_3(\lambda) \ \cong \ \cdots \ \cong \ E^{*,*}_{(d-1)(n-1)+1}(\lambda)
	\qquad\text{and}\qquad
	E^{*,*}_{(d-1)(n-1)+2}(\lambda) \ \cong \ \cdots \ \cong \ E^{*,*}_{\infty}(\lambda), 
	\end{equation}
	while for $k\geq 2$
	\begin{equation}
	\label{eq : ss - 02}
	E^{*,*}_2(\lambda) \ \cong \ E^{*,*}_3(\lambda) \ \cong \ \cdots \ \cong \ E^{*,*}_{(d-1)\big(n-\tfrac{n}{p}\big)+1}(\lambda).
	\end{equation}
	
\begin{figure}[h]
\begin{center}
\includegraphics[width=0.97\textwidth]{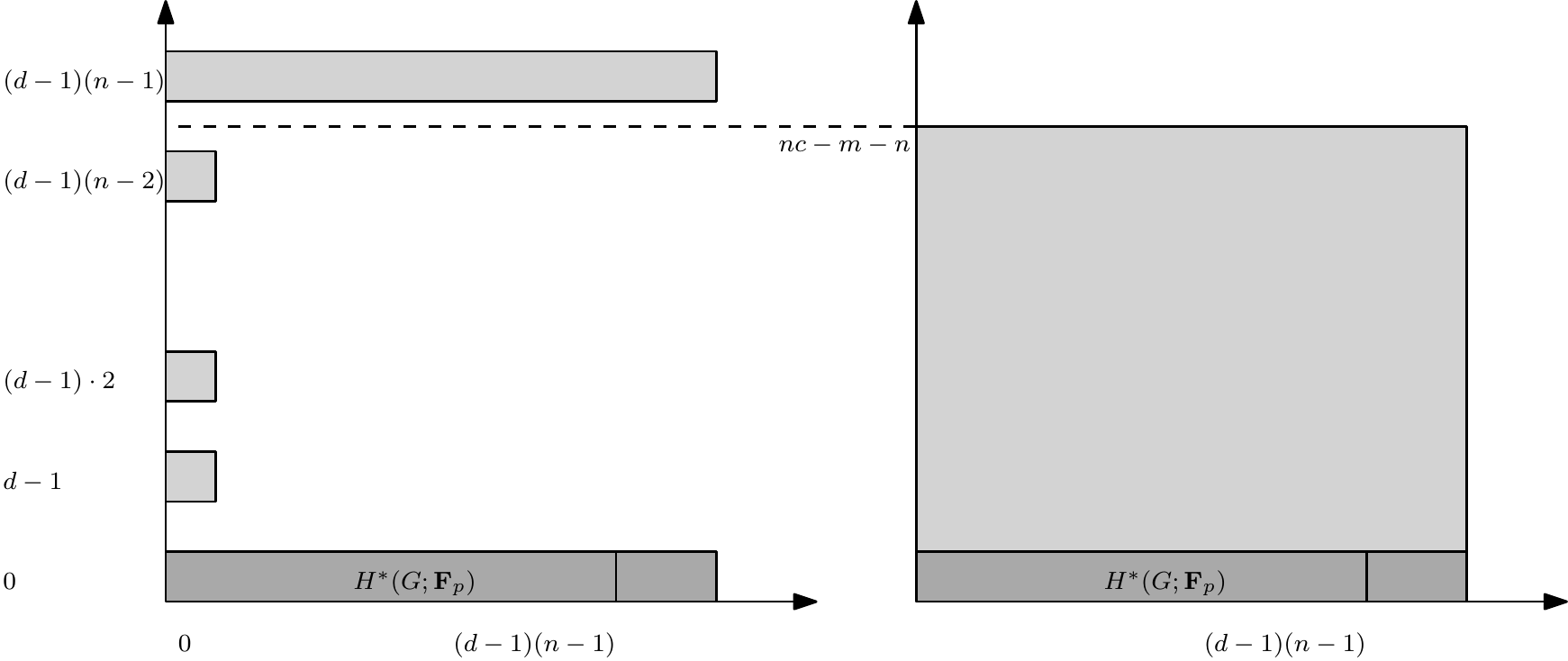}
\caption{An illustration of $E_2^{*,*}(\lambda)$ and $E_2^{*,*}(\rho)$ in the case when $n=p$ is a prime.}
\label{fig:poset}
\end{center}
\end{figure}
	
	\medskip
	In the second step we consider the Serre spectral sequence of the fibration
	\[
	\xymatrix{
	X\ar[r] & EG\times_G X\ar[r] & BG
	}
	\]
	whose $E_2$-term is given by
	\[
	E^{i,j}_2(\rho)=H^i(BG;\mathcal{H}^j(X;\F_p))\cong H^i(G;H^j(X;\F_p)).
	\]
	We conclude the proof by considering two separate cases.
	
	\medskip
	{\rm (a)} Let $\hat{1}\in Q'$, or equivalently $nc-m-2n+1\geq 0$.
	Then the simplicial complex $X=\Delta(Q'{\setminus}\{\hat{1}\})* \Delta(\varphi^{-1}(\{\hat{1}\}))$ is at most $(nc-m-n)$-dimensional, implying that
	$E_2^{i,j}(\rho)=0$ for all $j\geq nc-m-n+1$.
	Consequently all differentials $r_r$ for $r\geq nc-m-n+2$ vanish and so
	\begin{equation}
	\label{eq : ss - 03}
	E_{nc-m-n+2}^{i,j}(\rho) \ \cong \ E_{nc-m-n+3}^{i,j}(\rho) \ \cong \ \cdots \ \cong \ E_{\infty}^{i,j}(\rho). 
	\end{equation}
	 Next, since the path-connected simplical complex $X$ does not have fixed points with respect to the action of the elementary abelian group $G$, a consequence of the localization theorem \cite[Cor.\,1, p.\,45]{Hsiang1975} implies that $H^*(G;\F_p)\cong E^{*,0}_2(\rho)\not\cong E^{*,0}_{\infty}(\rho)$.
	 Having in mind \eqref{eq : ss - 03} we conclude that 
    \begin{equation*}
	H^*(G;\F_p) \ \cong \ E^{*,0}_2(\rho) \ \not\cong \ E_{nc-m-n+2}^{*,0}(\rho).
	\end{equation*}
	For our proof, without loss of generality, we can assume that
       \begin{equation}
	\label{eq : ss - 05}
     H^*(G;\F_p) \ \cong \ E^{*,0}_2(\rho) \ \cong \ E_{nc-m-n+1}^{*,0}(\rho) \ \not\cong \ E_{nc-m-n+2}^{*,0}(\rho).
    \end{equation}
    Now, from the assumption on $m$, we deduce that for $k=1$
    \[
    (d-1)(n-1)+1\geq nc-m-n+2,
    \]
    and for $k\geq 2$
     \[
    (d-1)\big(n-\frac{n}{p}\big)+1\geq nc-m-n+2.
    \]
	Hence the fact that $E^{*,0}_2(g)=\id$, in combination with relations \eqref{eq : ss - 01}, \eqref{eq : ss - 02} and \eqref{eq : ss - 05}, yields a contradiction: the homomorphism $E^{*,0}_{nc-m-n+2}(g)$ sends the zero to a non-zero element.
	This concludes the proof of the theorem in the case when $nc-2n+1\geq m$.
	
	\medskip
	{\rm (b)} Let $\hat{1}\notin Q'$, or equivalently $nc-m-2n+1< 0$.
	The simplicial complex $X=\Delta(Q')$ is at most $(n-2)$-dimensional.
	Hence, $E_2^{i,j}(\rho)=0$ for all $j\geq n-1$, and 
	\begin{equation}
	\label{eq : ss - 030}
	E_{n}^{i,j}(\rho) \ \cong \ E_{n+1}^{i,j}(\rho)\ \cong \ \cdots \ \cong \ E_{\infty}^{i,j}(\rho). 
	\end{equation}
	The simplical complex $X$ is path-connected  and without fixed points with respect to the action of the elementary abelian group $G$.
	Consequence of the localization theorem \cite[Cor.\,1, p.\,45]{Hsiang1975} implies that $H^*(G;\F_p)\cong E^{*,0}_2(\rho)\not\cong E^{*,0}_{\infty}(\rho)$.
	From \eqref{eq : ss - 030} we have that 
    \begin{equation*}
	H^*(G;\F_p) \ \cong \ E^{*,0}_2(\rho) \ \not\cong \ E_{n}^{*,0}(\rho).
	\end{equation*}
	For our proof, without loss of generality, we can assume that
       \begin{equation}
	\label{eq : ss - 050}
     H^*(G;\F_p) \ \cong \ E^{*,0}_2(\rho) \ \cong \ E_{n-1}^{*,0}(\rho) \ \not\cong \ E_{n}^{*,0}(\rho).
    \end{equation}
    Now, we need that for $k=1$
    \[
    (d-1)(n-1)+1\geq n,
    \]
    and for $k\geq 2$
    \[
    (d-1)\big(n-\frac{n}{p}\big)+1\geq n
    \]
    is satisfied. 
   	Indeed, these conditions are satisfied for $d \geq 2, p\geq 2$ and $n=p^k$.
	Thus, the fact that $E^{*,0}_2(g)=\id$ with \eqref{eq : ss - 01}, \eqref{eq : ss - 02} and \eqref{eq : ss - 050} gives a contradiction: the homomorphism $E^{*,0}_{n}(g)$ sends the zero to a non-zero element.
	This concludes the proof of the theorem in the case when $nc-2n+1< m$.	
\end{proof}

The previous proof can also be phrased in the language of the iterated index theory introduced by Volovikov in \cite{Volovikov2000}.

\subsection{Non-existence of an $\sym_n$-equivariant map {$\emp(\mu,n) \longrightarrow \bigcup\widetilde{\A}(m,n,c)$}}
\label{sec:Non-existence of equivariant maps II}

Motivated by Theorem \ref{thm : CS-TM scheme sum of measures}, in this section we prove the (non-)existence of a continuous $\sym_n$-equivariant map
\[
	  \emp(\mu,n) \longrightarrow \bigcup\widetilde{\A}(m,n,c)
\]
for different values of parameters $d,m,n$ and $c$. Following the structure of Section \ref{sec:Non-existence of equivariant maps I}, we first prove a few auxilary lemmas in order to arrive to the topological result, Theorem \ref{thm : CS-TM scheme sum of measures}, at the end of this section. 

\medskip

Recalling that a subspace of $\emp(\mu,n)$ consisting only of regular convex partitions can be identified with the configuration space $\conf(\R^d,n)$, see \cite[Sec.\,2]{Blagojevic2014} for more details, we obtain the following lemma.

\begin{lemma}
	\label{lemma : alpha tilde}
	There exists an $\sym_n$-equivariant map
	\[
	\widetilde{\alpha}\colon \conf(\R^d,n) \longrightarrow \emp(\mu,n).
	\]
\end{lemma}

\medskip

Denote by $\widetilde{P}=P(\widetilde{\A})$ the intersection poset of the affine arrangement $\widetilde{\A}$. 
Its elements are given by 
\[
\widetilde{p}_\Lambda := \bigcap_{(i,I) \in \Lambda} \widetilde{L}_{i,I} = 
\big\{ (y_{jk}) \in \widetilde{V} \subseteq \R^{m \times n} : y_{ji}=0, \text{ for all } 1 \leq i \leq n \text{ and } j\in I_i\big\},
\]
where $\Lambda \subseteq [n] \times  \binom{[m]}{m-c+1}$ and $I_i:= \bigcup_{(i,I)\in \Lambda} I$. 
An element $\widetilde{p}_{\Lambda}$ can also be seen as an $m \times n$ matrix $(a_{jk}) $, where $a_{jk}=0$ if and only if $j\in I_k$.

\medskip

Next we consider a $\widetilde{P}$-diagram $\widetilde{\C}$ determined by the arrangement $\widetilde{\A}=\widetilde{\A}(m,n,c)$. 
More precisely, we define $\widetilde{\C}(\widetilde{p}_\Lambda):= \widetilde{p}_\Lambda$ and 
$\widetilde{\C}(\widetilde{p}_{\Lambda'}\supseteq \widetilde{p}_{\Lambda''})\colon \widetilde{p}_{\Lambda''}\longrightarrow \widetilde{p}_{\Lambda'}$ to be the inclusion.  
The Equivariant Projection Lemma \cite[Lem.\,2.1]{Sundaram1997} implies the following.
 
\begin{lemma}
	\label{lemma : beta tilde}
	The projection map  
	\[
	\hocolim_{\widetilde{P}}\widetilde{\C} \longrightarrow \colim_{\widetilde{P}}\widetilde{\C}  = \bigcup\widetilde{\A}
	\]
	is an $\sym_n$-equivariant homotopy equivalence.
	In particular, there exists an $\sym_n$-equivariant map
	\[
	\widetilde{\beta}\colon \bigcup\widetilde{\A}\longrightarrow \hocolim_{\widetilde{P}}\widetilde{\C}.
	\]
\end{lemma}

\medskip
Recall that $Q$ denotes the face poset of an $(n-1)$-dimensional simplex, and define a map $\widetilde{\varphi}: \widetilde{P} \to Q$ by 
\[
\widetilde{\varphi}(\widetilde{p}_\Lambda) := \big\{i \in [n] : (i,I) \in \Lambda \text{ for some } I \subseteq [m]\big\}.
\]
Additionally, denote the poset $\widetilde{\varphi}(\widetilde{P}) \subseteq Q$ by $Q'$. 
Note that if $q,r \in Q$ are such that $q \in Q'$ and $r \leq q$, then $r$ is also an element of $Q'$. 
In particular, if $q=\hat{1}$ is the maximal element of $Q$ and $q\in Q'$, then $Q'=Q$.

\medskip
Let $\widetilde{\D}$ be the homotopy pushdown of the diagram $\widetilde{\C}$ along the map $\widetilde{\varphi}$ over $ Q'$. 
This means that
\begin{eqnarray*}
\widetilde{\D}(q):=\hocolim_{\widetilde{\varphi}^{-1}( Q'_{\geq q})} \widetilde{\C}|_{\widetilde{\varphi}^{-1} ( Q'_{\geq q})} \ \simeq \ \Delta(\widetilde{\varphi}^{-1}( Q'_{\geq q})) 
\end{eqnarray*}
for $q \in  Q'$, and the map $\widetilde{\D}(q\geq r)\colon \widetilde{\D}(q)\longrightarrow \widetilde{\D}(r)$
is the corresponding inclusion for every $q\geq r$ in $ Q'$. 
Once more, the Homotopy Pushdown Lemma \cite[Prop.\,3.12]{Welker1999} adapted to equivariant setting yields the following fact. 

\begin{lemma}
	\label{lemma : gamma tilde}
	There is an $\sym_n$-equivariant homotopy equivalence 
	\[
	\hocolim_{ Q'}\widetilde{\D} \longrightarrow \hocolim_{\widetilde{P}}\widetilde{\C} .
	\]
	In particular, there exists an $\sym_n$-equivariant map
	\[
	\widetilde{\gamma}\colon \hocolim_{\widetilde{P}}\widetilde{\C}\longrightarrow \hocolim_{ Q'}\widetilde{\D}.
	\]
\end{lemma}

\medskip
Finally, we consider another $ Q'$-diagram $\widetilde{\E}$ by setting for $q\in  Q'$ that
\begin{eqnarray*}
\widetilde{\E}(q) = \begin{cases}
		\Delta(\widetilde{\varphi}^{-1}(\{\hat{1}\})), & \textrm{ if } q=\hat{1}\in  Q' \textrm{ is the meximum of } Q, \\
		\pt, & \textrm{ otherwise,}
		\end{cases}
\end{eqnarray*}
and the map $\widetilde{\E}(q\geq r)$ to be the constant map for every $q\geq r$ in $ Q'$.
Similarly as we have done it in Section \ref{sec:Non-existence of equivariant maps I}, we define a morphism of diagrams $(\widetilde{\Psi},\widetilde{\psi})\colon \widetilde{\D}\longrightarrow\widetilde{\E}$, where $\widetilde{\psi}\colon  Q'\longrightarrow  Q'$ is the identity map, and $\widetilde{\Psi}(q)\colon \widetilde{\D}(q)\longrightarrow \widetilde{\E}(q)$ is the identity map when $q=\hat{1}$ is the maximal element in $Q$, and constant map otherwise.
Since the morphism $(\widetilde{\Psi},\widetilde{\psi})$ of diagrams induces an $\sym_n$-equivariant map between associated homotopy colimits, we have established the following.

\begin{lemma}
	\label{lemma : delta tilde}
	There exists an $\sym_n$-equivariant map
	\[
	\widetilde{\delta}\colon \hocolim_{Q'}\widetilde{\D}\longrightarrow \hocolim_{ Q'}\widetilde{\E}.
	\]
\end{lemma}

\medskip	
Just like in Section \ref{sec:Non-existence of equivariant maps I}, the final lemma describes the $\hocolim_{Q'}\widetilde{\E}$ up to an $\sym_n$-equivariant homotopy.

\begin{lemma}
\label{lemma : hocolim E tilde}
~~
\begin{compactenum}[\rm\qquad (i)]
\item If $\hat{1}\in Q'$, that is if $Q'=Q$, then there exists an $\sym_n$-equivariant homotopy equivalence
\[
\hocolim_{Q}\widetilde{\E}\simeq \Delta (Q{\setminus}\{\hat{1}\}) * \Delta(\widetilde{\varphi}^{-1}(\{\hat{1}\}))
\]
where $\hat{1}$ is the maximaum of $Q$, and $\dim\big(  \Delta(\widetilde{\varphi}^{-1}(\hat{1}))\big) =nc-n-\max\{m,n\}$.
In particular, there exists an $\sym_n$-equivariant map
\[
\widetilde{\eta}\colon \hocolim_{Q}\widetilde{\E}\longrightarrow \Delta (Q{\setminus}\{\hat{1}\}) * \Delta(\widetilde{\varphi}^{-1}(\{\hat{1}\})).
\]

\item If $\hat{1}\notin Q'$ then there exists an $\sym_n$-equivariant homotopy equivalence
\[
\hocolim_{Q'}\widetilde{\E}\simeq \Delta (Q'),
\]
where $\dim ( \Delta (Q'))\leq n-2$.
In particular, there exists an $\sym_n$-equivariant map
\[
\widetilde{\eta}\colon \hocolim_{Q'}\widetilde{\E}\longrightarrow \Delta (Q').
\]
\end{compactenum}
\end{lemma}
\begin{proof}
The proof of the claim {\rm (ii)} is identical to the proof of the second part of Lemma \ref{lemma : hocolim E}. For the claim {\rm (i)} it suffices to compute the dimension of the simplicial complex $\Delta(\widetilde{\varphi}^{-1}(\{\hat{1}\}))$, since the rest of the proof follows the lines of the proof of the first part of Lemma \ref{lemma : hocolim E}.

The elements of the poset $\widetilde{\varphi}^{-1}(\{\hat{1}\})$ are presented by matrices $\widetilde{p}_\Lambda = (a_{jk})$ that contain zeros in every column. The partial order is given by 
\[
\widetilde{p}_\Lambda \leq \widetilde{p}_{\Lambda'}  \ \Longleftrightarrow \ (\forall j\in [m]) \ (\forall k \in [n]) \ a_{jk}=0 \Rightarrow a'_{jk}=0,
\]
where $\widetilde{p}_\Lambda = (a_{jk})$ and $\widetilde{p}_{\Lambda'} = (a'_{jk})$ are elements of the poset $\widetilde{\varphi}^{-1}(\{\hat{1}\}) \subseteq \widetilde{P}$. 
Maximal chains in $\widetilde{\varphi}^{-1}(\{\hat{1}\})$ can be obtained by removing zeros one by one from a matrix that represents a maximal element, taking care of the fact that every column has to contain at least $m-c+1$ zeros. 
The maximal elements are presented by matrices that have at most $n-1$ zeros in each row, and at most $m-1$ zeros in each column. 
Thus, maximal elements are presented by matrices with $mn-\max(m,n)$ zeros. 
The minimal elements, on the other hand, are presented by matrices that contain $n(m-c+1)$ zeros. 
Therefore, the dimension of $\Delta(\widetilde{\varphi}^{-1}(\hat{1}))$ is $nc-n-\max\{m,n\}\geq 0$.
Since $c\geq 2$, this implies that $n(c-1) \geq m$.
\end{proof}

\medskip
Now we are ready to prove the central result about the non-existence of an $\sym_n$-equivariant map $\emp(\mu,n) \longrightarrow \bigcup\widetilde{\A}$.
 
\begin{theorem}
\label{thm:no map sum of measures}
Let $d\geq 2$, $m\geq 2$, and $c\geq 2$ be integers, and let $n=p^k$ be a prime power.
If 
\begin{compactenum}[\rm\qquad (a)]
\item $n(c-1) \geq m$ and $\max\{m,n\}\geq n(c-d)+\tfrac{dn}{p}-\tfrac{n}{p}+n$, or
\item $n(c-1) < m$,
\end{compactenum}
 then there is no $\sym_n$-equivariant map
\begin{equation}
	\label{eq : map that should not exist - 01 tilde}
	 \emp(\mu,n) \longrightarrow \bigcup\widetilde{\A}(m,n,c),
\end{equation}
where $\mu = \mu_1 + \dots + \mu_m$ is the sum of $m$  finite absolutely continuous measures on $\R^d$, and the affine arrangement $\widetilde{\A}(m,n,c)$ is as defined in \eqref{eq:A sum of measures}.
\end{theorem}

\begin{proof}
It is not surprising that this proof will follow the lines of the proof of Theorem \ref{thm:no map one measure}.
Let $n=p^k$ be a prime power and denote by $G\cong (\Z/p)^k$ a subgroup of the symmetric group $\sym_n$ given by the regular embedding $\mathrm{(reg)} \colon G \longrightarrow \sym_n$.

\medskip
The proof will proceed by contradiction. Therefore, assume that $\widetilde{f}\colon \emp(\mu,n) \longrightarrow \bigcup\widetilde{\A}(m,n,c)$ is a continuous $\sym_n$-equivariant map.
From Lemmas \ref{lemma : alpha tilde}, \ref{lemma : beta tilde}, \ref{lemma : gamma tilde}, \ref{lemma : delta tilde} and \ref{lemma : hocolim E tilde} we again get a composition of $\sym_n$-equivariant maps
\[
\xymatrix{
 \emp(\mu,n) \ar[r]^-{\widetilde{f}} &  \bigcup\widetilde{A}\ar[r]^-{\widetilde{\beta}} & \hocolim_{\widetilde{P}} \widetilde{\C}\ar[r]^-{\widetilde{\gamma}} & \hocolim_{ Q'}\widetilde{\D}\ar[r]^-{\widetilde{\delta}} & \hocolim_{ Q'}\widetilde{\E}\ar[d]_{\widetilde{\eta}}  \\
\conf(\R^d,n)\ar[u]^{\widetilde{\alpha}}\ar@{-->}[rrrr]^-{\widetilde{g}:=\widetilde{\eta}\circ\widetilde{\delta}\circ\widetilde{\gamma}\circ\widetilde{\beta}\circ \widetilde{f}\circ\widetilde{\alpha}} &  &  &  &  \widetilde{X}
}
\]
where
\[
\widetilde{X}:=
\begin{cases}
	\Delta( Q{\setminus}\{\hat{1}\}) * \Delta(\widetilde{\varphi}^{-1}(\{\hat{1}\})), & \text{if } \hat{1}\in  Q',\\
	\Delta( Q'),  &\text{if } \hat{1}\notin  Q'.
	\end{cases} 
\]
It suffices to show that the map $\widetilde{g}$ cannot exist, since that would contradict the existence of the map $\widetilde{f}$.
Actually, we will prove here that there is no continuous $G$-equivariant map
\begin{equation}
		\label{eq : map that should not exist - 02 tilde}
		\conf(\R^d,n)\longrightarrow \widetilde{X}.
\end{equation}

\medskip
We start by considering a continuous $\sym_n$ and also $G$-equivariant map $\widetilde{g}\colon \conf(\R^d,n)\longrightarrow \widetilde{X}$.
It induces a morphism between Borel construction fibrations:
	\[
	\xymatrix{
	EG\times_G  \conf(\R^d,n)\ar[rr]^{\id\times_G \widetilde{g}}\ar[d]_{\lambda} & & EG\times_G \widetilde{X}\ar[d]_{\widetilde{\rho}} \\
	BG\ar[rr]^-{\id}  & & BG,
	}
	\]
which furthermore induces a morphism between the corresponding Serre spectral sequences 
\[
E^{*,*}_*(\widetilde{g})\colon E^{*,*}_*(\widetilde{\rho})\longrightarrow E^{*,*}_*( \lambda ).
\]
Like in the proof of Theorem \ref{thm:no map one measure}, we use the fact that $E^{*,0}_2(\widetilde{g})=\id$. 
Next we analyse the morphism $E^{*,*}_*(\widetilde{g})$.
Since the spectral sequence $E^{*,*}_*(\lambda)$ was already described in the proof of Theorem \ref{thm:no map one measure}, we concentrate here on the spectral sequence $E^{*,*}_*(\widetilde{\rho})$.

\medskip

The Serre spectral sequence of the fibration
	\[
	\xymatrix{
	\widetilde{X}\ar[r] & EG\times_G \widetilde{X}\ar[r] & BG
	}
	\]
	has the $E_2$-term given by
	\[
	E^{i,j}_2(\widetilde{\rho})=H^i(BG;\mathcal{H}^j(\widetilde{X};\F_p))\cong H^i(G;H^j(\widetilde{X};\F_p)).
	\]
	In order to conclude the proof we consider two separate cases.
	
	\medskip
	{\rm (a)} Let $\hat{1}\in Q'$ and let $m$ satisfy the condition of the theorem. 
	Then $n(c-1)\geq m$, so the simplicial complex $\widetilde{X}=\Delta(Q{\setminus}\{\hat{1}\})* \Delta(\widetilde{\varphi}^{-1}(\{\hat{1}\}))$ is at most $(nc-\max\{m,n\}-1)$-dimensional. This implyies that
	$E_2^{i,j}(\widetilde{\rho})=0$ for all $j\geq nc-\max\{m,n\}$,
	 and consequently,
	\begin{equation}
	\label{eq : ss - 03 tilde}
	E_{nc-\max\{m,n\}+1}^{i,j}(\widetilde{\rho}) \ \cong \ E_{nc-\max\{m,n\}+2}^{i,j}(\widetilde{\rho}) \ \cong \ \cdots \ \cong \ E_{\infty}^{i,j}(\widetilde{\rho}). 
	\end{equation}
	
	Once more a consequence of the localization theorem \cite[Cor.\,1, p.\,45]{Hsiang1975} implies that $H^*(G;\F_p)\cong E^{*,0}_2(\widetilde{\rho})\not\cong E^{*,0}_{\infty}(\widetilde{\rho})$, because the path-connected simplical complex $\widetilde{X}$ does not have fixed points with respect to the action of the elementary abelian group $G$.	 
	 Having in mind \eqref{eq : ss - 03 tilde} we conclude that 
    \begin{equation*}
	H^*(G;\F_p) \ \cong \ E^{*,0}_2(\widetilde{\rho}) \ \not\cong \ E_{nc-\max\{m,n\}+1}^{*,0}(\widetilde{\rho}).
	\end{equation*}
	For our proof, without loss of generality, we can assume that
       \begin{equation}
	\label{eq : ss - 05 tilde}
     H^*(G;\F_p) \ \cong \ E^{*,0}_2(\widetilde{\rho}) \ \cong \ E_{nc-\max\{m,n\}-2}^{*,0}(\widetilde{\rho}) \ \not\cong \E_{nc-\max\{m,n\}-1}^{*,0}(\widetilde{\rho}).
    \end{equation}
    Now, the assumption on $m$ and $n$, means for $k=1$
    \[
    (d-1)(n-1)+1\geq nc-\max\{m,n\}+1,
    \]
    and for $k\geq 2$
     \[
    (d-1)\big(n-\frac{n}{p}\big)+1\geq nc-\max\{m,n\}+1.
    \]
    Therefore, the relations \eqref{eq : ss - 01}, \eqref{eq : ss - 02} and \eqref{eq : ss - 05 tilde}, together with the fact that $E^{*,0}_2(\widetilde{g})=\id$, yield a contradiction: the homomorphism $E^{*,0}_{nc-\max\{m,n\}+1}(\widetilde{g})$ sends the zero to a non-zero element.
	This concludes the proof of the theorem in the case when $n(c-1)\geq m$.
	
	\medskip
	{\rm (b)} Let $\hat{1}\notin Q'$, or equivalently $n(c-1)< m$.
	The simplicial complex $\widetilde{X}=\Delta(Q')$ is at most $(n-2)$-dimensional, by Lemma \ref{lemma : hocolim E tilde}(ii).
	Thus, $E_2^{i,j}(\widetilde{\rho})=0$ for all $j\geq n-1$, and furthemore 
	\begin{equation}
	\label{eq : ss - 030 tilde}
	E_{n}^{i,j}(\widetilde{\rho})\ \cong \ E_{n+1}^{i,j}(\widetilde{\rho}) \ \cong \ \cdots \ \cong \ E_{\infty}^{i,j}(\widetilde{\rho}). 
	\end{equation}
	For the same reason as above, we have $H^*(G;\F_p)\cong E^{*,0}_2(\widetilde{\rho})\not\cong E^{*,0}_{\infty}(\widetilde{\rho})$.
	This fact combined with the isomorphisms in \eqref{eq : ss - 030 tilde} yields that  
    \begin{equation*}
	H^*(G;\F_p) \ \cong \ E^{*,0}_2(\widetilde{\rho}) \ \not\cong \ E_{n}^{*,0}(\widetilde{\rho}).
	\end{equation*}
	Again, without loss of generality, we can assume that
       \begin{equation}
	\label{eq : ss - 050 tilde}
     H^*(G;\F_p) \ \cong \ E^{*,0}_2(\widetilde{\rho}) \ \cong \ E_{n-1}^{*,0}(\widetilde{\rho}) \ \not\cong \ E_{n}^{*,0}(\widetilde{\rho}).
    \end{equation}
    In order to complete the proof we need that for $k=1$
    \[
    (d-1)(n-1)+1\geq n,
    \]
    and for $k\geq 2$
     \[
    (d-1)\big(n-\frac{n}{p}\big)+1\geq n.
    \]
    Indeed, both of these inequalities are satisfied, thus the fact that $E^{*,0}_2(\widetilde{g})=\id$ with \eqref{eq : ss - 01}, \eqref{eq : ss - 02} and \eqref{eq : ss - 050 tilde} gives a contradiction: the homomorphism $E^{*,0}_{n}(\widetilde{g})$ sends the zero to a non-zero element.
	This concludes the proof of the theorem in the case when $n(c-1)< m$.	
\end{proof}

\section{Proofs}
\label{sec:proofs}

Finally, in this section proofs of theorems \ref{thm:geometric}, \ref{thm:one measure}, \ref{thm:sum of measures}, \ref{th:referee's}, \ref{thm:optimal}, \ref{th:referee's - 02}, and \ref{th:referee's - 03} will be presented. 
The proofs of theorems \ref{thm:geometric} and \ref{thm:optimal} are completely geometric and they do not involve any topological methods. The proofs of theorems \ref{th:referee's}, \ref{th:referee's - 02}, and \ref{th:referee's - 03} rely on earlier results that use much simpler topological tools than those involved in theorem \ref{thm:one measure} and \ref{thm:sum of measures}.
In particular, the proofs of theorems \ref{thm:geometric}, \ref{th:referee's}, \ref{thm:optimal}, \ref{th:referee's - 02}, and \ref{th:referee's - 03} are independent from Sections \ref{sec:From partitions to equivariant topology} and \ref{sec:Non-existence of equivariant maps}. 
On the other hand, the proofs of Theorem \ref{thm:one measure} and Theorem \ref{thm:sum of measures} heavily depend on the topological results from the previous sections.

\subsection{Proof of Theorem \ref{thm:geometric}} 

Let $d\geq 2$, $m\geq2$, $n\geq 2$ and $c\geq d$ be integers such that $m\geq n(c-d)+d$.
Since the measures $\mu_1, \dots, \mu_m$ are positive, finite, and absolutely continuous with respect to the standard Lebesgue measure, the interiors of their supports are non-empty.
For  each $1 \leq j \leq m$, choose a point $v_j \in \int(\supp (\mu_j))$ in the interior of the support of the measure $\mu_j$.
Set $V:=\{v_1, \dots , v_m\}$. 
Perturb the points $v_1, \dots, v_m$ if necessary, so that they are in general position, that is no $d+1$ of them lie in the same affine hyperplane. 
The set $P:=\conv(V)$ is a $d$-dimensional simplicial polytope in $\R^d$. 
Choose any $(d-2)$-dimensional face $F$ of the polytope $P$.
Since $P$ is simplicial the face $F$ is a simplex and so has $d-1$ vertices that belong to $V$.

\medskip
First, we look for an affine hyperplane $H$ in $\R^d$ with the properties that:
\begin{compactitem}[\qquad---]
\item $F\subseteq H$,
\item $\# (V\cap H)=\# (V\cap F)+1=d$, and
\item $\# (V\cap \int (H^+))=c-d$.
\end{compactitem}
The hyperplane $H$ cuts $\R^d$ into two half-spaces closed half-spaces $H^+$ and $H^-$ such that $H^+$ has positive measure with respect to at least $c$ of the measures $\mu_1, \dots, \mu_m$, because it intersects interiors of supports of at least $c$ measures.
Such a hyperplane exists. 
Indeed, since $F$ is a face of $P$ there exists a supporting hyperplane $H'$ for $F$, that is a hyperplane that contains the face $F$ and one of its closed half-spaces contains $P$, see Figure \ref{fig:R^2 - 1}. 
Rotate the hyperplane $H'$ around the $(d-2)$-dimensional subspace spanned by $F$ to get the hyperplane $H$ such that there are exactly $c-d$ points of $V$ in $\int(H^+)$, and furthermore there is one additional point from $V{\setminus}F$ on $H$.
Since the affine span of $F$ is a hyperplane in $H$ the additional point, say $w$, lies in the relative interior of one of the half-hyperplanes of $H$ determined by $F$.
Denote the half-hyperplanes of $H$ determined by $F$ with $K_0$ and $K_{m-c+1}$ such that $w\in \relint(K_0)$.
In particular, $K_0\cup K_{m-c+1}=H$.

\medskip
The set  $V^-:=V \cap H^-$ is of cardinality $m-c$. 
Consider all half-hyperplanes whose boundary is the affine span of $F$ and contains a point of $V^-$ in its relative interior. 
Since the set $V$ is in general position, there are exactly $m-c$ such half-hyperplanes. 
Label them $K_1, \dots, K_{m-c}$ in order, starting from the half-hyperplane that forms the smallest angle with the half-hyperplane $K_0$, as illustrated in Figure \ref{fig:R^2 - 3}. 
The affine span of $F$ and the half-hyperplanes
\[
K_0, \ K_{c-d}, \ K_{2(c-d)}, \ \dots, \ K_{(n-2)(c-d)}, \ K_{m-c+1}
\]
define an $n$-fan whose every region intersects interiors of the supports of at least $c$ of the measures $\mu_1, \dots, \mu_m$.
Indeed, the region defined by
\begin{compactitem}[\qquad---]
\item $H^+$, or equivalently by $K_{m-c+1}$ and $K_0$ contains exactly $c$ points from $V$,
\item $K_0$ and $K_{c-d}$ contains exactly $d-1+c-d+1=c$ points from $V$,
\item $K_{c-d}$ and $K_{2(c-d)}$ contains exactly $d-1+c-d+1=c$ points from $V$,
\item \dots
\item $K_{(n-2)(c-d)}$ and $K_{m-c+1}$ contains $d-1+m-c-(n-2)(c-d)+1=(m-nc+nd-d)+c\geq c$ points from $V$.
\end{compactitem}
An example for $d=2$, $n=5$ and $c=4$ is shown in Figure \ref{fig:R^2 - 4}.

\medskip
Thus, we constructed a convex partition of $\R^d$ such that each piece of the partition has a positive measure with respect to at least $c$ of the measures $\mu_1, \dots, \mu_m$.
This concludes the proof of the theorem. \qed

\begin{figure}[h]
\begin{center}
\subfigure[The face $F$ and the hyperplane $H'$.]{\label{fig:R^2 - 1}\includegraphics[width=0.45\textwidth]{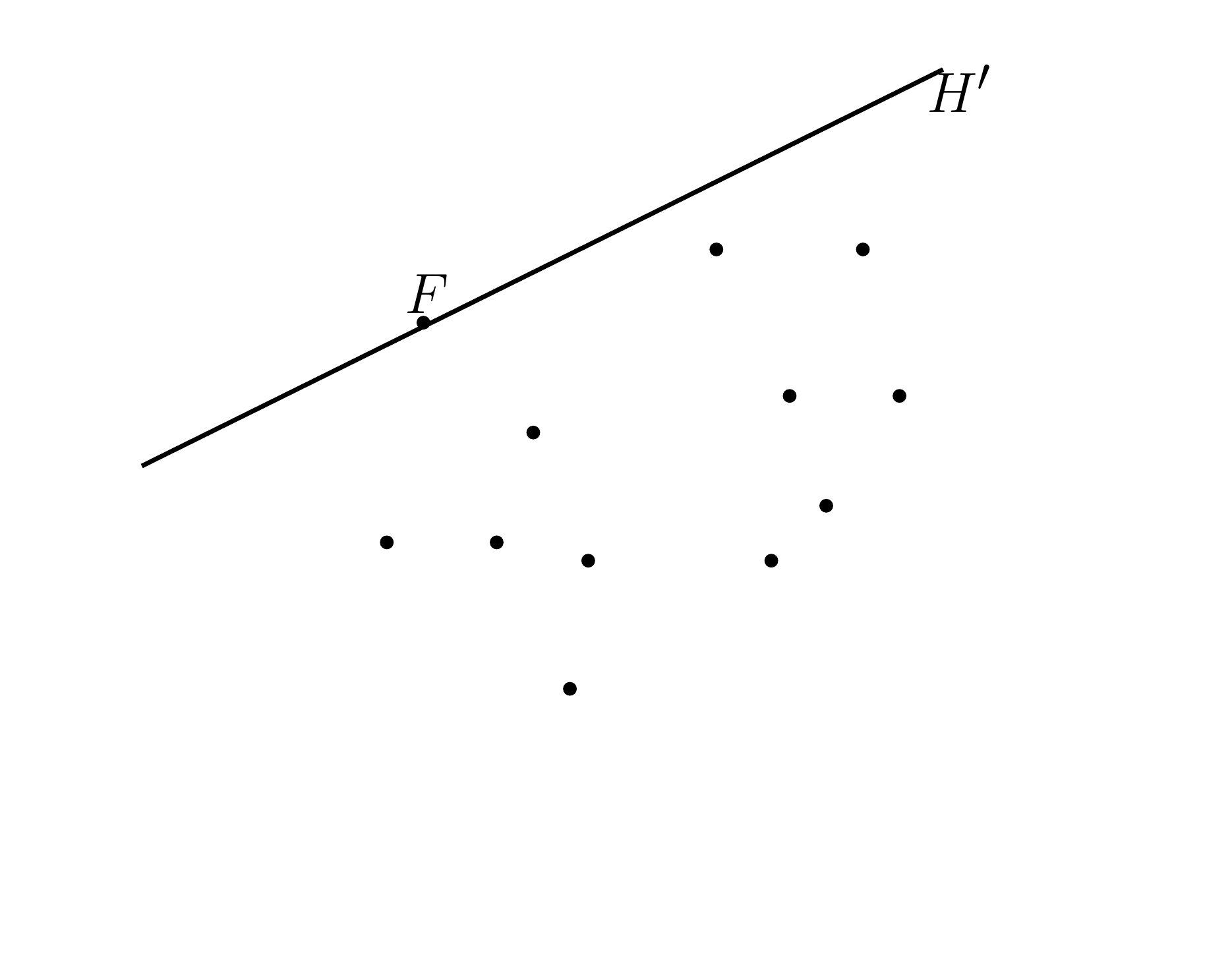}}
\subfigure[The face $F$, the point $w$ and the final position of the hyperplane $H'=H$.]{\label{fig:R^2 - 2}\includegraphics[width=0.45\textwidth]{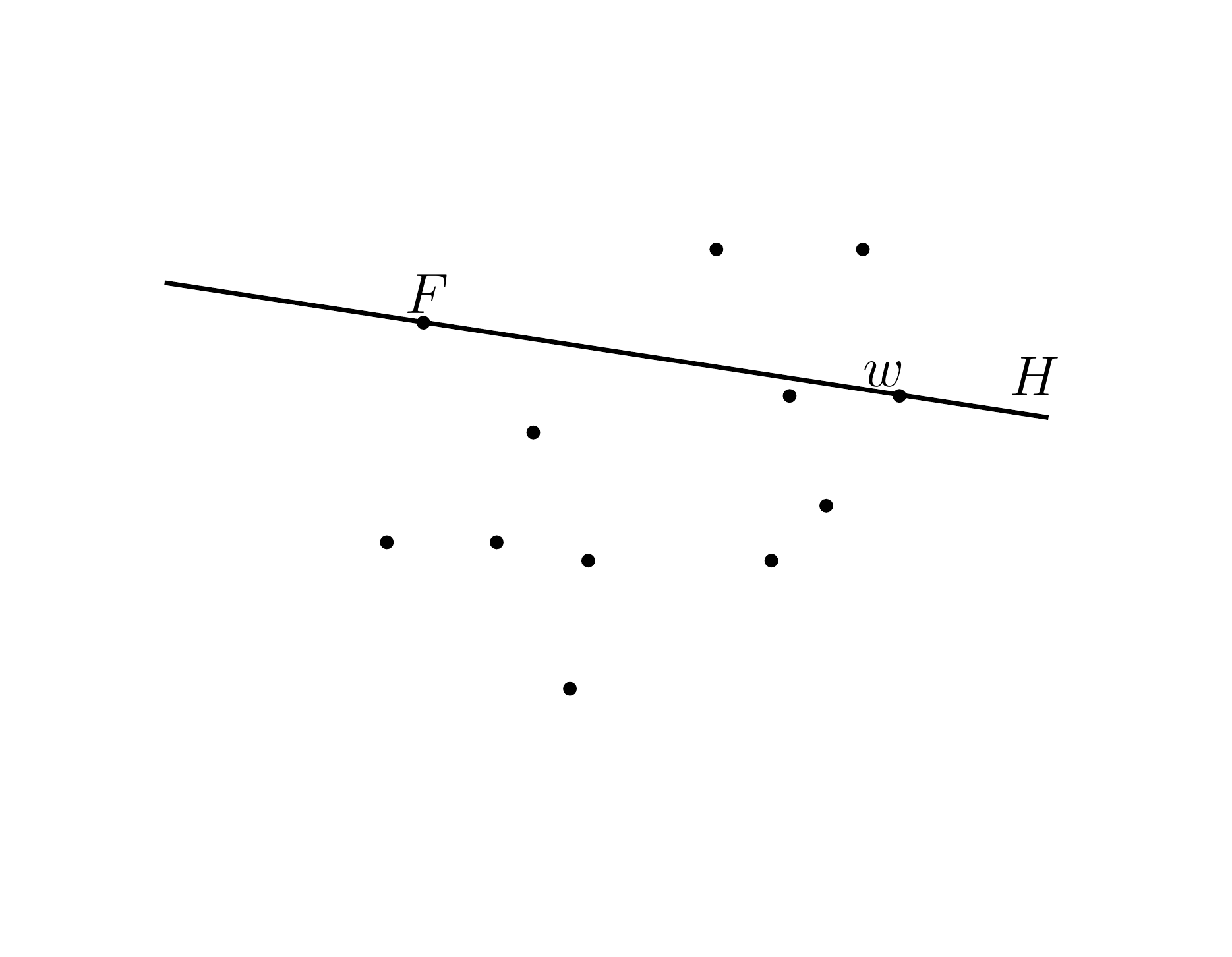}}
\subfigure[Labeling of the half-hyperplanes in half-space $H^-$.]{\label{fig:R^2 - 3}\includegraphics[width=0.45\textwidth]{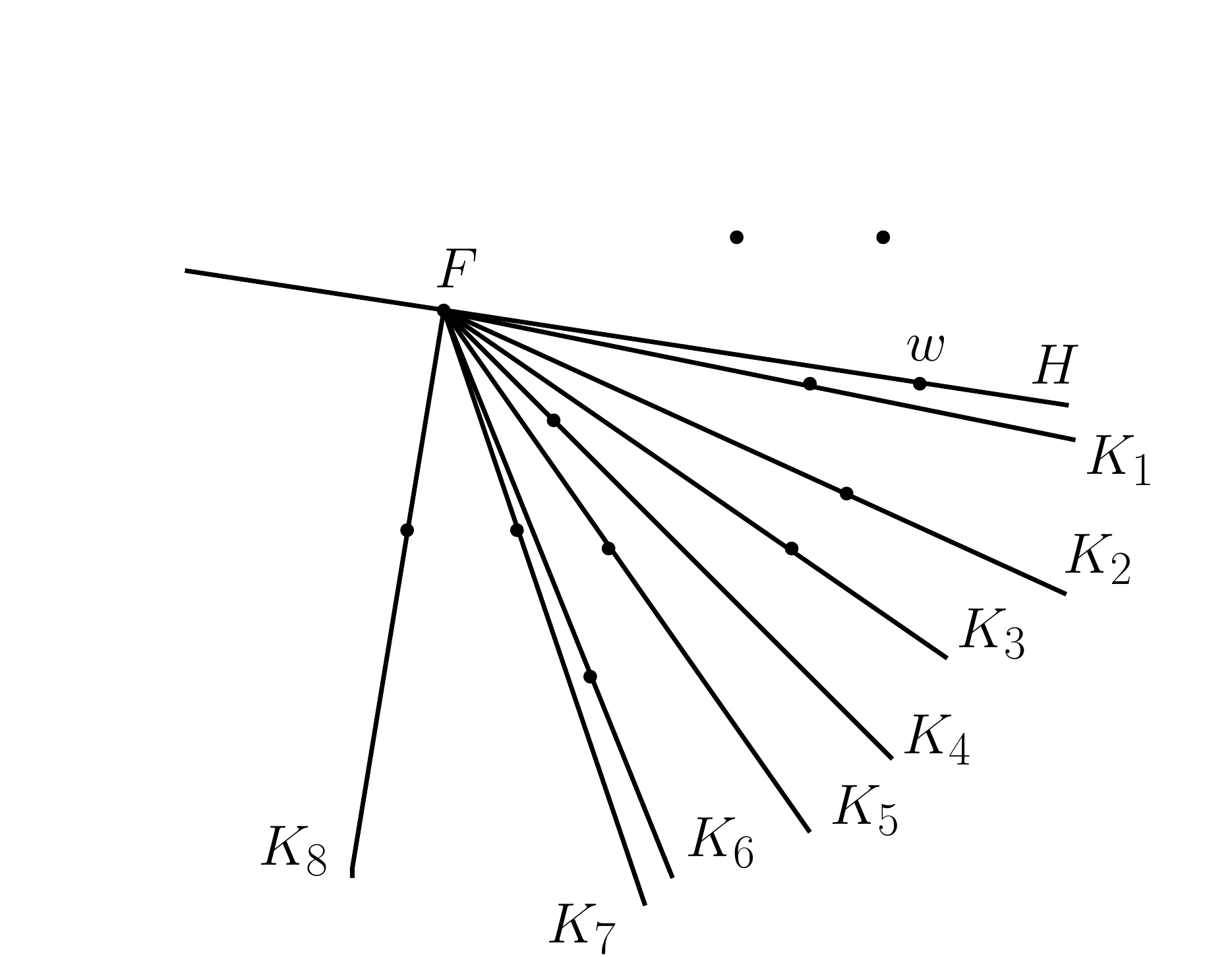}}
\subfigure[A $5$-fan that partitions $\R^2$ into convex pieces so that each piece has positive measure with respect to at least $4$ measures. ]{\label{fig:R^2 - 4}\includegraphics[width=0.45\textwidth]{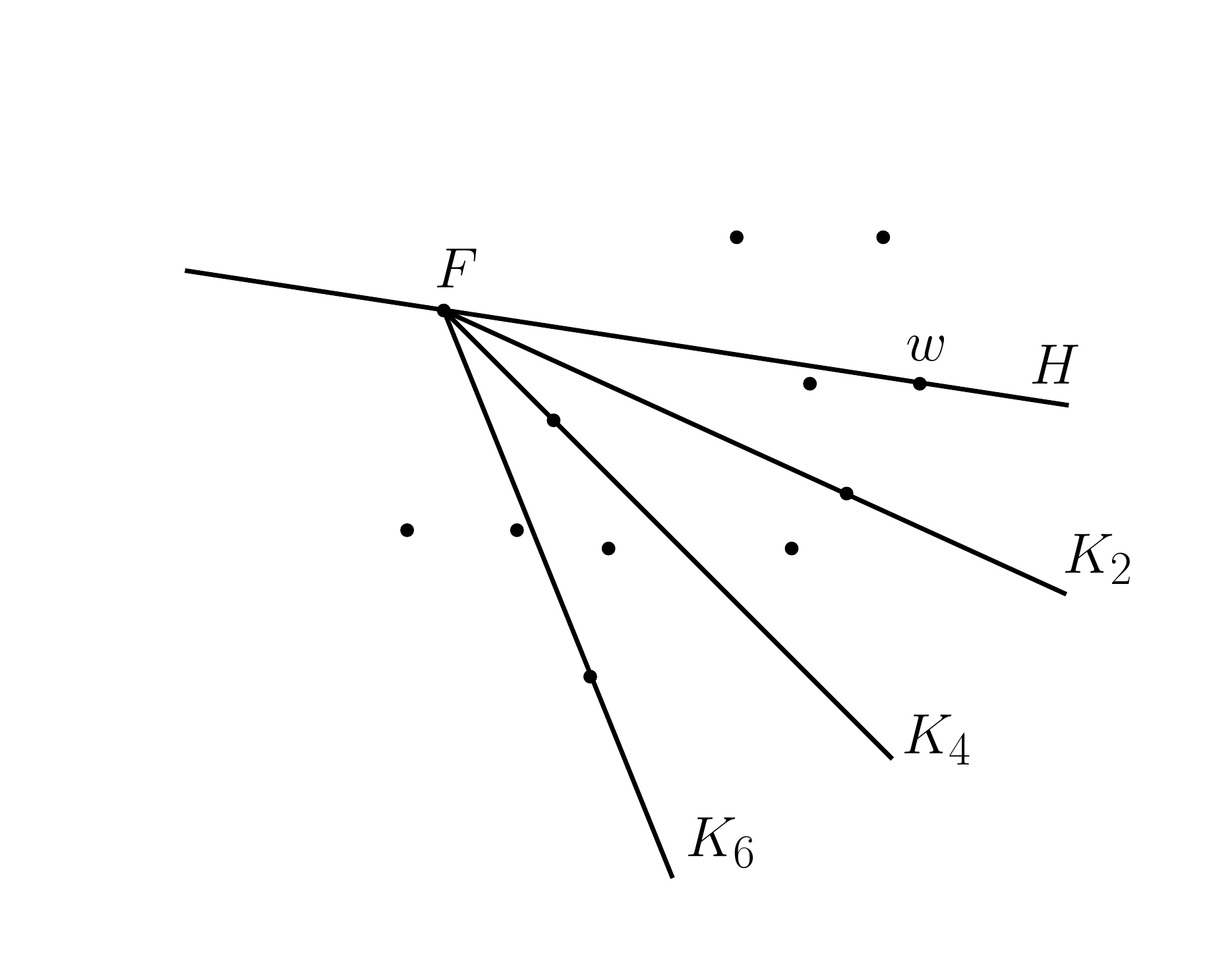}}
\caption{An example of a fan partition of $\R^2$ for $n=5$ and $c=4$.}
\end{center}
\end{figure}

\begin{remark}
As a consequence of the previous proof, there is a convex partition $(C_1, \dots, C_n)$ of $\R^d$, such that each piece $C_i$ has positive measure with respect to at least $c$ of the measures $\mu_1, \dots, \mu_m$, and additionally all pieces $C_1, \dots, C_n$ have positive measure with respect to $d-1$ measures $\mu_{j_1}, \dots, \mu_{j_{d-1}}$, where $F=\conv\{v_{j_1}, \dots, v_{j_{d-1}}\}$ and $v_{j_k} \in \relint(\supp(\mu_{j_k}))$, for every $1 \leq k \leq d-1$. In contrast to the statement of Theorem \ref{thm:one measure}, we cannot guarantee an equipartition, and we cannot choose which measure will be contained in each piece.
\end{remark}

\subsection{Proof of Theorem \ref{thm:one measure}}
Let $d\geq 2$, $m\geq 2$, and $c\geq 2$ be integers, and let $n\geq 2$ be a prime power.
Under the assumptions of the theorem on $m$, Theorem \ref{thm:no map one measure} yields the non-existence of an $\sym_n$-equivariant map
\[
\emp(\mu_m,p) \longrightarrow \bigcup\A(m,n,c).
\]
Consequently, Theorem \ref{thm : CS-TM scheme one measure} implies that for every collection of $m$ measures $\mu_1, \dots, \mu_m$ in $\R^d$ there exists a convex partition $(C_1, \dots, C_n)$ of $\R^d$ with the property that  each of the subsets $C_i$ has positive measure with respect to at least $c$ of the measures $\mu_1, \dots, \mu_m$. 
In other words, 
\[
\# \big\{ j : 1\leq j\leq m,\,\mu_j(C_i)>0\big\} \geq c
\] 
for every $1\leq i \leq n$.\qed

\begin{remark}
In order to prove the non-existence of the $G$-equivariant map $f: \emp(\mu_m,n) \longrightarrow \bigcup \A$, one could directly try show that there is no $G$-equivariant map $\conf(\R^d,n) \longrightarrow \bigcup \A = \colim_{P(\A)} \C$. 
However, since the dimension of the order complex of $P(\A)$ is 
\[
\dim(\Delta(P(\A))) = nc-n-c,
\]
this method proves Theorem \ref{thm:one measure} only for $c \leq d$, which follows directly from the result of Sober\'on \cite{Soberon2012}.
\end{remark}

\subsection{Proof of Theorem \ref{thm:sum of measures}}
\label{sec:proof referee - 01}

Let $d\geq 2$ and $c\geq 2$ be integers, let $n\geq 2$ be a prime power and let $m \geq 2$ be an integer that satisfies the conditions of the theorem.
Theorem \ref{thm:no map sum of measures} yields the non-existence of an $\sym_n$-equivariant map
\[
\emp(\mu,p) \longrightarrow \bigcup\widetilde{\A}(m,n,c),
\]
and Theorem \ref{thm : CS-TM scheme sum of measures} implies that for every collection of $m$ measures $\mu_1, \dots, \mu_m$ in $\R^d$ there exists a convex partition $(C_1, \dots, C_n)$ of $\R^d$ with the property that  each of the subsets $C_i$ has positive measure with respect to at least $c$ of the measures $\mu_1, \dots, \mu_m$. 
In other words, 
\[
\# \big\{ j : 1\leq j\leq m,\,\mu_j(C_i)>0\big\} \geq c
\] 
for every $1\leq i \leq n$.

\subsection{The first proof of Theorem \ref{th:referee's}}
Let $d\geq 2$, $n\geq 2$ and $c\geq d$ be integers, and set $m= n(c-d)+d$.
Consider $m$ measures $\mu_1, \dots, \mu_m$ on $\R^d$ which are positive, finite and absolutely continuous with respect to the standard Lebesgue measure.
Thus, the interiors of their supports are non-empty.
For  each $d \leq j \leq m$, choose a point $v_j \in \int(\supp (\mu_j))$ in the interior of the support of the measure $\mu_j$.
Now consider the set $V:=\{v_d, \dots , v_m\}$ as a point measure.

\medskip
The result of Sober\'on \cite{Soberon2012} applied to the collection of measures  $\mu_1, \dots, \mu_{d-1},V$ guaranties the existence of a convex partition $(C_1,\dots,C_n)$ of $\R^d$ that equiparts $\mu_1, \dots, \mu_{d-1}$ and in addition the point measure $V$. 
For the point measure $V$ it means that
\[
\# (V\cap C_i)\geq  \Big\lceil\frac{m-d+1}{n} \Big\rceil \geq \Big\lceil\frac{n(c-d)+d-d+1}{n} \Big\rceil = c-d +1,
\]
for every $1\leq i\leq n$. 
Consequently, 
\[
\# \big\{ j : 1\leq j\leq m,\,\mu_j(C_i)>0\big\} \geq c
\] 
for all $1\leq i \leq n$.

\subsection{The second proof of Theorem \ref{th:referee's}}
\label{sec:proof referee - 02}

Let $d\geq 2$, $n\geq 2$ and $c\geq d$ be integers, and let $m= n(c-d)+d$.
Consider $m$ measures $\mu_1, \dots, \mu_m$ on $\R^d$ which are positive, finite and absolutely continuous with respect to the standard Lebesgue measure.
We introduce a new measure  $\nu$ by
	\[
	\nu (A) = \sum_{i=d}^m \frac{\mu_i(A)}{\mu_i(\R^d)},
	\]
for $A\subseteq\R^d$ a measurable set.
Then $\nu(\R^d) = m-(d-1) = n(c-d)+1$, and each of the measures $\mu_d, \dots, \mu_m$ can contribute at most $1$ to $\nu(A)$ for any $A \subseteq \R^d$.  
Therefore, if $\nu(A) > k$ for some non-negative integer $k$, then $A$ must have positive size in at least $k+1$ of the measures from $\mu_d, \ldots, \mu_{m}$.  

\medskip
Again, the result of Sober\'on \cite{Soberon2012} applied to the collection of measures  $\mu_1, \dots, \mu_{d-1},\nu$ guaranties the existence of a convex partition $(C_1,\dots,C_n)$ of $\R^d$ that equiparts measures $\mu_1, \dots, \mu_{d-1}$ as well as the measure $\nu$. 
In particular,
\[
	\nu(C_i) = \frac{n(c-d)+1}{n} > c - d \geq 0
\]
for every $1\leq i\leq n$. 
Therefore, each of the pieces $C_i$ of the partition has positive size in at least $c-d+1$ measures among $\mu_d, \ldots, \mu_m$.
Having in addition exactly $\frac{1}{n}$ fraction of each of the measures $\mu_1, \ldots, \mu_{d-1}$ each convex piece $C_i$ of the partition has positive size in at least $c$ measures among $\mu_1, \ldots, \mu_m$.\qed

\subsection{Proof of Theorem \ref{thm:optimal}}

First, we prove that $m=n(c-d)+d$ is the optimal bound in Theorem \ref{thm:optimal} for $d=1$.  
It suffices to consider $n(c-1)+1$ measures on $\R$, each concentrated near a point, so that the supports are pairwise disjoint.  
Each convex partition is given by a family of intervals whose interiors are pairwise disjoint.
Thus, and consecutive intervals can share the support of at most one measure.  
A careful counting, as illustrated in Figure \ref{fig:optimal-01}, shows that $n(c-1)+1$ measures are needed for each interval to intersect the support of $c$ measures.

\begin{figure}[ht]
\begin{center}
\includegraphics[width=0.96\textwidth]{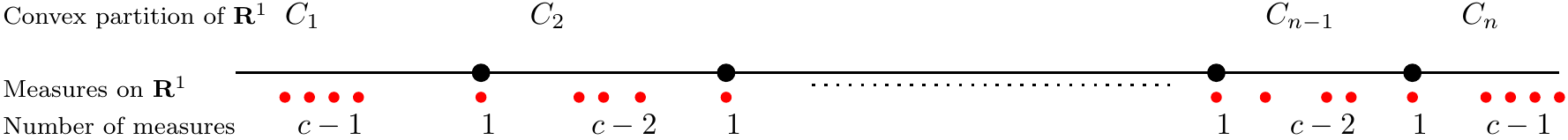}
\caption{Optimality of the bound in Theorem \ref{thm:optimal} in the case $d=1$.}
\label{fig:optimal-01}
\end{center}
\end{figure}

\medskip
For the case $d\geq 2$ we use the following construction illustrated in Figure \ref{fig:optimal-02}.  
Let $e_1, \ldots, e_{d}$ denote the vector of the standard basis, and let $\ell$ be the line $e_d + \operatorname{span}(e_{d-1})$.  
We place $m - d +1$ measure $\mu_d, \ldots, \mu_m$ concentrated in points of $\ell$, as in the case of dimension $1$.
Now consider a regular simplex in the affine subspace $-e_d + \operatorname{span}(e_1, \ldots, e_{d-2})$, of dimension $d-2$, centered around $-e_d$.  
Place $d-1$ measures $\mu_1, \ldots, \mu_{d-1}$, each concentrated near a vertex of the simplex we just constructed.  
We ask from our partition into convex subsets to equiparts the measures $\mu_1, \ldots, \mu_{d-1}$.
Thus, each piece of the convex partition has positive size in $\mu_1, \ldots, \mu_{d-1}$.
In oder for the interior of their convex hulls to be disjoint, they have to intersect the remaining measures $\mu_d, \ldots, \mu_{m}$ as in the case of dimension $1$. 
Since they must each intersect at least $c-(d-1)$ measure on the line $\ell$, and $\ell$ has $m-(d-1)$ measures, the case of dimension one implies that
\[
m-(d-1) \geq  n(c-(d-1)-1) +1 \qquad\Longleftrightarrow\qquad m\geq n(c-d) + d.
\]
This concludes the proof of the theorem. \qed

\begin{figure}[ht]
\begin{center}
\includegraphics[width=0.45\textwidth]{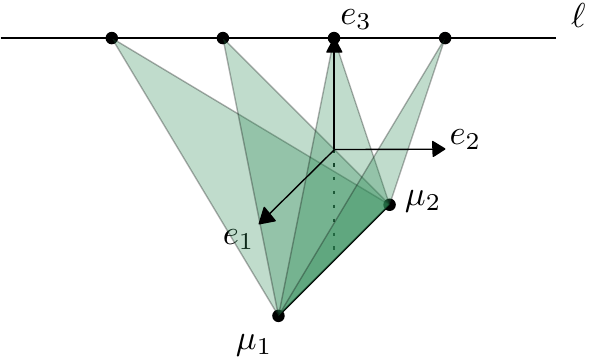}
\caption{Optimality of the bound in Theorem \ref{thm:optimal} in the case $d\geq2$.}
\label{fig:optimal-02}
\end{center}
\end{figure}

\subsection{Proof of Theorem \ref{th:referee's - 02}}\label{sub:Proof of Thm. Referee 02}
For the proof of the theorem we first derive the following version of the pigeonhole principle.

\begin{claim}\label{claim}
Let $m\geq 1$ and $1\leq r\leq m$ be integers, and let $0\leq x_1,\dots,x_m\leq 1$  and $0\leq \varepsilon\leq 1$ be real numbers.
If 
\begin{equation}
	\label{eq:assumption}
	x_1+\dots+x_m\geq r-1+\varepsilon(m-r+1),
\end{equation}
then there exist at least $t$ indices $1\leq i_1 < \dots <i_{t}\leq m$ such that 
\[
\min\{x_{i_1},\dots,x_{i_{r}}\}\geq\varepsilon.
\]
\end{claim}
\begin{proof}
Suppose that the claim does not hold.
Without lost of generality we can assume that 
\[
0\leq x_m \leq\dots\leq x_r\leq x_{r-1}\leq\dots\leq x_1\leq 1.
\]
Then
\[
 x_m \leq\dots\leq x_r<\varepsilon.
\]
Consequently, 
\[
x_1+\dots+x_m=(x_1+\dots+x_{r-1})+(x_{r}+\dots +x_m)<r -1+\varepsilon(m-r+1).
\] 	
We reached a contradiction with the assumption \eqref{eq:assumption}.
\end{proof}

\medskip 
Now we proceed with the proof of Theorem \ref{th:referee's - 02}.
Let $d\geq 2$, $n\geq 2$ and $c\geq d$ be integers, and let $m= n(c-d)+d$.
Consider $m$ measures $\mu_1, \dots, \mu_m$ on $\R^d$ which are positive, finite and absolutely continuous with respect to the standard Lebesgue measure.
We partition our set of $m$ measures into $d$ subsets $I_1, \ldots, I_d$ such that  
\begin{compactitem}[\qquad---]
\item  $c=\sum_{k=1}^d r_k\geq d$ for some positive integers $r_1,\dots, r_d$, and
\item $\# I_k = m_k = n(r_k -1) +1$ for all $1\leq k\leq d$. 
\end{compactitem}
For each $1\leq k\leq d$ we define the measure $\nu_k$ on $\R^d$ by
\[
\nu_k(A) = \sum_{\mu \in I_k} \frac{\mu(A)}{\mu (\R^d)},
\]
where $A\subseteq\R^d$  is a measurable set.
Consequently, $\nu_k (\R^d) = \#I_k = n(r_k -1) +1$. 

\medskip
Using the result of Sober\'on \cite{Soberon2012} applied to the collection of measures $\nu_1, \ldots, \nu_d$ we get a convex partition $(C_1,\dots,C_n)$ of $\R^d$ with the property that 
\[
\nu_k (C_i) = \frac{\nu_k (\R^d)}{n} = r_k -1 + \frac{1}{n}
\]
for every $1\leq k\leq d$ and every $1\leq i\leq n$.

\medskip
Now fix $k$ and $i$, and consider  $m_k$ real numbers  $\frac{\mu(C_i)}{\mu (\R^d)}$, $\mu\in I_k$, in the interval $[0,1]$.
Since 
\[
 \sum_{\mu \in I_k} \frac{\mu(C_i)}{\mu (\R^d)}=\nu_k (C_i)= r_k -1 + \frac{1}{n},
\]
we can apply Claim \ref{claim} to get at least $r_k$ numbers out of $\frac{\mu(C_i)}{\mu (\R^d)}$, $\mu\in I_k$, greater or equal than $\varepsilon_k=\frac{1}{n(m_k-r_k+1)}$.
Thus, there is at least $r_k$ measures $\mu$ in the set $I_k$ such that 
\[
\mu(C_i)\geq \varepsilon_k\mu(\R^d)=\frac{1}{n(m_k-r_k+1)}\mu(\R^d),
\]
where $1\leq k\leq d$.
Now, we choose each $r_k$ to be either $\lfloor \frac{c}{d} \rfloor$ or $\lceil \frac{c}{d} \rceil$ in such a way that $\sum_{k=1}^d r_k = c$.
Consequently, we obtain a uniform lower bound for each fraction $\varepsilon_k$ that we were seeking for.

\medskip
Hence, each of the subsets $C_1,\dots,C_n$ has at least $\varepsilon$ fraction in at least $c$ of the measures $\mu_1, \dots, \mu_m$, and the proof of the theorem is concluded.\qed

\subsection{Proof of Theorem \ref{th:referee's - 03}}
The proof of the theorem proceeds along the lines of the proof of Theorem \ref{th:referee's - 02} presented in Section \ref{sub:Proof of Thm. Referee 02}.
Let $d\geq 2$, $m\geq 2$, $n\geq 2$, and $c\geq 2d$ be integers, and let $0<\alpha<\frac1{n}$ be a real number.
Assume that
\[
m \ge (c-d)\Big( \frac{1-\alpha}{\frac{1}{n}-\alpha}\Big) +d - 1.
\] 

\medskip
Let $r_1, \ldots, r_d$ be integers such that $r_k \ge 2$ for all $1 \le k \le d$, and let $r_1 + \dots + r_d = c$.  
For each $1 \le k \le d$ we set
\[
m_k = \Big\lceil (r_k - 1) \Big( \frac{1-\alpha}{\frac1n - \alpha}\Big)\Big\rceil >0.
\]
Then $m_k < (r_k -1) \big( \frac{1-\alpha}{\frac1n - \alpha}\big) + 1$ for $1 \le k \le d$, and consequently
\[
m_1 + \dots + m_d < (c-d)\Big( \frac{1-\alpha}{\frac1n - \alpha}\Big)+d \le m + 1.
\]
Since $m, m_1, \dots, m_d$ are all integers, the previous inequality implies that 
\[
m_1 + \dots + m_d\leq m.
\]  

\medskip
Now, out of thee set of $m$ measures $\{\mu_1,\dots,\mu_m\}$ we can choose $d$ non-empty disjoint subsets of measures $I_1, \dots, I_d$ such that $\#I_k = m_k$, $1 \leq k \leq d$.  
For each $1 \leq k \leq d$ we define, as before, the measure $\nu_k$ on $\R^d$ by
\[
\nu_k(A) = \sum_{\mu \in I_k} \frac{\mu(A)}{\mu (\R^d)},
\]
where $A \subseteq \R^d$ is a measurable set.
Hence, 
\[
\nu_k (\R^d) = \# I_k = m_k = \Big\lceil (r_k - 1) \Big( \frac{1-\alpha}{\frac1n - \alpha}\Big)\Big\rceil \geq (r_k -1)\Big( \frac{1-\alpha}{\frac{1}{n}-\alpha}\Big).
\]  
The inequality 
\[
m_k \geq (r_k -1)\Big( \frac{1-\alpha}{\frac{1}{n}-\alpha}\Big)
\]
 can be rearranged into
\[
\frac{m_k}{n} \ge r_k -1 + \alpha (m_k - r_k + 1).
\]

\medskip
The result of Sober\'on \cite{Soberon2012} applied on the collection of measures $\nu_1, \ldots, \nu_d$ gives a convex partition $(C_1,\dots,C_n)$ of $\R^d$ with the property that 
\[
\nu_k (C_i) = \frac{\nu_k (\R^d)}{n} = \frac{m_k}{n} \ge r_k - 1 + \alpha(m_k - r_k + 1).
\]
for every $1\leq k\leq d$ and every $1\leq i\leq n$.  Now, fix any $1 \le i \le n$ and $1 \le k \le d$.  Consider the $m_k$ numbers of the form $\frac{\mu(C_i)}{\mu(\R^d)}$ for $\mu \in I_k$.  From Claim \ref{claim}, it follows that at leas $r_k$ of those numbers are at least $\alpha$.  If we repeat this for every $k$, we obtain that $C_i$ has a fraction $\alpha$ of at least $r_1 + \dots + r_k = c$ measures, as desired.

If the number $\frac{1-\alpha}{\frac{1}{n}-\alpha}$ is an integer, then we set $m_k=(r_k -1)\left( \frac{1-\alpha}{\frac{1}{n}-\alpha} \right)$. 
Consequently, we only require the bound
\[
m \ge (c-d)\Big(\frac{1-\alpha}{\frac{1}{n}-\alpha}\Big)
\]
on the number of measures to derive the theorem. \qed

\begin{remark}
One can observe that the extra ``$+(d-1)$'' on the bound of $m$ is not always needed, but the precise value would then depend on a careful choice of the parameters $r_1, \dots, r_d$, which would require a case-by-case analysis.  
Some cases yield very clean bounds.  
For example, if $\alpha = \frac{1}{2n-1}$, we are getting many measure that have more than half of what would be the optimal bound.  
Indeed, for $m = 2n(c-d)$ measures in $\R^d$, there exists a partition of $\R^d$ into $n$ convex parts such that each part has at least $\frac{1}{2n-1}$ fraction of at least of $c$ measures.
\end{remark}

\end{document}